%% file: main.tex
\documentclass{amsart}

\usepackage{amsmath,amsthm,amsfonts,amscd,amssymb} 
\usepackage{verbatim}      	% Allows quoting source with commands.
\usepackage{url}		% Allows good typesetting of web URLs.
\usepackage{epic,eepic,latexsym}
\usepackage{graphicx}
\usepackage{color}
\usepackage{lcd}
\usepackage{mathrsfs}
\usepackage[centering,text={15.5cm,22cm},
        marginparwidth=20mm,dvips]{geometry}
\usepackage{verbatim}
\usepackage{hyperref} %[linktocpage] removed
\usepackage{lscape}
\usepackage{tabularx}
\usepackage{marginnote}
\usepackage[all,color]{xy}
\UseCrayolaColors
\usepackage{enumitem}
\usepackage{tikz-cd}
\usepackage[acronym,sort=use]{glossaries}
\usepackage{multicol}

\DeclareMathOperator{\Hom}{Hom}

\DeclareMathOperator{\Pic}{Pic}

\DeclareMathOperator{\Sym}{Sym}

\DeclareMathOperator{\codim}{codim}

\DeclareMathOperator{\rank}{rank}

\DeclareMathOperator{\Spec}{Spec}
\DeclareMathOperator{\id}{Id}

\DeclareMathOperator{\val}{val}

\DeclareMathOperator{\TV}{TV}

\DeclareMathOperator{\trop}{trop}

\DeclareMathOperator{\Mult}{Mult}

\DeclareMathOperator{\vir}{vir}
\DeclareMathOperator{\ev}{ev}
\DeclareMathOperator{\Aut}{Aut}

\DeclareMathOperator{\sat}{sat}

\DeclareMathOperator{\Sing}{Sing}

\DeclareMathOperator{\ns}{ns}

\DeclareMathOperator{\LC}{LC}
\DeclareMathOperator{\C}{C}
\DeclareMathOperator{\LL}{L}
\DeclareMathOperator{\ov}{ov}
\DeclareMathOperator{\forget}{Forget}
\DeclareMathOperator{\pt}{pt}
\DeclareMathOperator{\GW}{GW}

\DeclareMathOperator{\codomain}{codomain}

\DeclareMathOperator{\inde}{index}
\DeclareMathOperator{\domain}{domain}
\DeclareMathOperator{\coker}{coker}
\DeclareMathOperator{\vdim}{vdim}

\DeclareMathOperator{\Stab}{Stab}

\let\?=\overline
\let\:=\colon
\let\bb=\mathbb

\let\rar=\rightarrow
\let\ra=\rightarrow

\let\f=\mathfrak
\let\s=\mathcal

\let\op=\operatorname

\let\wh=\widehat
\let\wt=\widetilde

\def\risom{\buildrel\sim\over{\smashedlongrightarrow}}
 \def\smashedlongrightarrow{\setbox0=\hbox{$\longrightarrow$}\ht0=1.25pt\box0}

\newcommand {\kk} {\Bbbk}
\newcommand {\A} {{\bf A}}

\newcommand {\PP} {{\bf P}}

\newcommand {\GG} {{\mathbb{G}}}
\newcommand {\QQ} {{\mathbb{Q}}}
\newcommand {\RR} {{\mathbb{R}}}

\newcommand {\ZZ} {{\mathbb{Z}}}
\newcommand {\NN} {{\mathbb{N}}}

\newcommand {\shM} {{\mathcal{M}}}

\theoremstyle{plain}% default
 \newtheorem{thm}{Theorem}[section]
 
 \newtheorem{lem}[thm]{Lemma}
  \newtheorem{prop}[thm]{Proposition}
  
   \newtheorem{cor}[thm]{Corollary}
	 \newtheorem{obs}[thm]{Observation}

\theoremstyle{definition}
 \newtheorem{dfn}[thm]{Definition}

 \newtheorem{eg}[thm]{Example}
 
 \newtheorem{lemdfn}[thm]{Lemma/Definition}

\theoremstyle{remark} 
 \newtheorem{rmk}[thm]{Remark}
  
  \newtheorem{ass}[thm]{Assumption}

\newenvironment{myproof}[1][\proofname]{\proof[#1]}{\endproof}

\input{glossary}

\makeglossaries

\title[Descendant log Gromov-Witten invariants for toric varieties and tropical curves]{Descendant log Gromov-Witten invariants for toric varieties and tropical curves}
\author{Travis Mandel}
\address{School of Mathematics\\
University of Edinburgh\\
Edinburgh EH9 3FD\\
UK}\email{Travis.Mandel{\char'100}ed.ac.uk}
\author{Helge Ruddat}
\address{JGU Mainz, Institut f\"ur Mathematik, Staudingerweg 9, 55128 Mainz, Germany}
\thanks{\tiny The first author was supported by the Center of Excellence Grant ``Centre for Quantum Geometry of Moduli Spaces'' from the Danish National Research Foundation (DNRF95), and later by the National Science Foundation RTG Grant DMS-1246989 and the Starter Grant “Categorified Donaldson-Thomas Theory” no.~759967 of the European Research Council.  The second author was partially supported by the DFG Emmy Noether grant RU 1629/4-1 and DFG SFB TR 45 and is grateful for hospitality at the IAS in Princeton}
\email{ruddat{\char'100}uni-mainz.de}

\date{\today}

\begin{document}

\maketitle
\begin{abstract}
Using degeneration techniques, we prove the correspondence of tropical curve counts and log Gromov-Witten invariants with general incidence and psi-class conditions in toric varieties for genus zero curves. For higher-genus situations, we prove the correspondence for the non-superabundant part of the invariant.  We also relate the log invariants to the ordinary ones, in particular explaining the appearance of negative multiplicities in the descendant correspondence result of Mark Gross.
\end{abstract}

\setcounter{tocdepth}{1}
\tableofcontents  
\vspace{-.9cm} %This is here to have nicer page wrapping for Thm 1.1
\section{Introduction}
In a pioneering work \cite{Mi}, Mikhalkin proved a correspondence of counts of algebraic curves in a toric surface with analogous counts of certain piecewise linear graphs called \emph{tropical curves}.
Around the same time, Siebert and Nishinou \cite{NS} used very different groundbreaking techniques (toric degenerations and log geometry) to prove a genus $0$ correspondence theorem in any dimension.  Building on this second approach, this article generalizes both these results in the following directions:
\begin{itemize}
\item to allow $\psi$-conditions on the curves (i.e.~to consider descendant invariants) because such conditions are particularly useful in various applications,
\item to generalize to all non-superabundant situations (no a priori genus or dimension restrictions),
\item to allow incidence conditions in the toric boundary for applications to non-toric situations like non-toric blow-ups or Calabi-Yau degenerations,
\item to allow arbitrary tropical cycles as incidence conditions, not just affine linear ones.
\end{itemize}
To clarify: genus zero situations are always non-superabundant, but superabundant curves may contribute when $g>0$, especially if we also have $\dim \geq 3$, and also for $\dim =2$ if $\psi$-classes are present (see Remark~\ref{SuperExamples}).  Even in these cases though, our result still gives the correspondence for the non-superabundant part of the invariant (in analogy to residual intersection theory).

Mikhalkin suggested in \cite[\S3]{Mi2} that $\psi$-class conditions manifest in tropical geometry as higher-valence conditions at markings.  Markwig-Rau \cite{MR} and M. Gross \cite{GrP2} proved genus $0$ descendant correspondence theorems involving higher-valent vertices for ordinary GW invariants of $\bb{P}^2$, later generalized by P. Overholser \cite{Over} also for $\bb{P}^2$. Note also \cite{Rau}'s work for $\bb{P}^1$ and $\bb{P}^1\times \bb{P}^1$.
Notably, all these correspondence results are for ordinary Gromov-Witten invariants. We discuss the difference between these and the log invariants below.

Concerning genus $0$ log invariants, while our paper was in progress, \cite{Ran} used ideas from tropical intersection theory to give a new proof of the Nishinou-Siebert result, and then A. Gross \cite{AGr} extended these methods to allow for gravitational ancestors (i.e., pullbacks of $\psi$-classes from $\?{\s{M}}_{0,n}$, which by our Prop. \ref{psihat} turn out to coincide with the usual descendant $\psi$-classes). 
The Nishinou-Siebert degeneration approach we employ is quite different from these works, in particular making the connection between individual stable maps and tropical curves very explicit.

To our knowledge, the only previous work on tropical $\psi$-classes in higher genus has been for one-dimensional targets with zero-dimensional incidence conditions, cf. \cite[\S 3]{CJMR} and our discussion in \S \ref{section-hurwitz}.  In fact, even without $\psi$-classes, we have not seen the full details of the Nishinou-Siebert approach worked out for higher-genus curves anywhere, although the relevant log deformation theory was worked out in \cite{Ni}. 

We note that boundary incidence conditions were previously investigated from the Nishinou-Siebert degeneration perspective in dimension $2$ in the appendix of \cite{GPS}.  The non-affine tropical incidence conditions which we explain in \S \ref{GenInc} were, to our knowledge, previously not considered from the degeneration perspective but are easy to handle in genus $0$ from the tropical intersection theory perspective employed in \cite{Ran} and \cite{AGr}.

For our algebraic counts we use log Gromov-Witten (GW) invariants (cf. \cite{GSlog}, \cite{AC}), although for the cases we consider this will turn out to coincide with a naive algebraic count (possibly with some fractional multiplicities when $g>0$, cf. Theorem~\ref{mainthmintro}).

We follow the degeneration approach of \cite{NS}, but in place of their log deformation theory we take advantage of the newer technology regarding the existence of a virtual fundamental class for the moduli stack of basic stable log maps and 
the invariance of log GW invariants under log-smooth deformations.  We provide a proof of this invariance in the appendix (Theorem~\ref{Invariance}), both for convenience and to bring to attention why one needs Lemma~\ref{lem-reg-emb}.
We explicitly describe a smooth open subspace of the expected dimension in the moduli stack of basic stable log maps to the central fiber. 
The intersection of all incidence and $\psi$-conditions has support in this open set and we explicitly describe this effective zero-cycle and its stratification by tropical curves.
As applications, we prove two results about the relationship of the log GW invariants to the ordinary ones 
(Theorems~\ref{thm-forgetlog-equal} and \ref{main-comparison-thm}).
We then show how our theorem in dimension one specializes to the known result that double Hurwitz numbers are descendant log/relative GW invariants (\cite{CJM, CJMR}, Theorem~\ref{H-GW}). As discussed in \S\ref{motivation-mirror}, the authors are working on further applications to problems in mirror symmetry.

\begin{figure}
\resizebox{0.4\textwidth}{!}{
\includegraphics{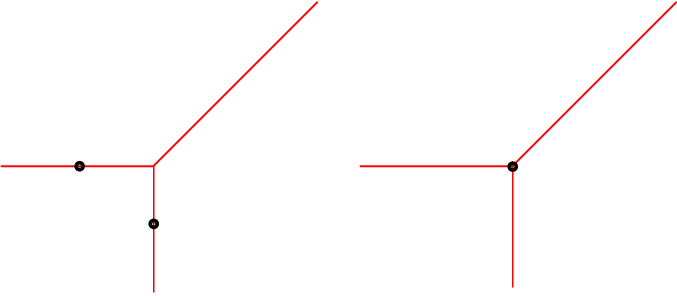}
}
\caption{Unique tropical line through two given points, respectively through a single point with a $\psi$-class condition.}
\label{trop-line}
\end{figure}
As a first simple example, consider the tropical analogue of a line in $\bb{P}^2$ meeting a point where it satisfies a $\psi$-class condition (equivalently in this case, where it has a specified tangent direction). This is the tropical line with incidence point in the trivalent vertex, see Figure~\ref{trop-line}. The reason from our point of view is as follows. Tropicalizing a stable log map yields the tropical curve as the dual intersection graph of the domain curve, and the valency of a vertex matches the number of special points (markings or nodes) of the corresponding curve component.  Higher valencies thus give curve components more moduli (we see in Lemma \ref{PsiGeomNew} that the $\psi$-classes can essentially be pulled back from the moduli of the corresponding components of the domain, see also Prop. \ref{psihat}), and since a $\psi$-condition cuts down such moduli, a minimal valency is necessary for such a curve with $\psi$-conditions imposed on markings on that component to have a non-zero contribution.

A significant difference between log and ordinary GW invariants is the fact that $\psi$-classes in ordinary GW theory can be much more complicated, leading for example to the possibility of negative counts.  A simple example is $\langle \ell\psi^2\rangle^{\bb{P}^2}_{0,\ell}$, i.e., the ordinary GW count of lines in $\bb{P}^2$ meeting a given generic line $\ell$ in a generic $\psi^2$-condition. 
Our tropical count ${}^{\operatorname{trop}}\langle \ell\psi^2\rangle^{\bb{P}^2}_{0,\ell}$ and hence our log invariant $\langle \ell\psi^2\rangle^{\bb{P}^2(\log D)}_{0,\ell}$ is $0$ here since satisfying the $\psi^2$ condition would require a four-valent vertex, and this is impossible for a tropical line in $\bb{P}^2$, see Fig.~\ref{trop-line}.  However, $\langle \ell\psi^2\rangle^{\bb{P}^2}_{0,\ell}=-3$. 

The discrepancy essentially arises because the log GW count considers extra markings at each point where the $\bb{P}^1$ meets the toric boundary $D$ of $\bb{P}^2$. By Thm.~\ref{main-comparison-thm}, this log GW invariant can be related to the usual one by removing these extra markings via the divisor equation. Using our correspondence theorem Thm~\ref{mainthmintro}, we obtain
$$0\underset{\hbox{\tiny valency}\atop{\hbox{\tiny argument}}}{=}
{}^{\operatorname{trop}}\langle \ell\psi^2\rangle^{\bb{P}^2}_{0,\ell}
\underset{\hbox{\tiny correspondence}\atop{\hbox{\tiny theorem \ref{mainthmintro}}}}{=}
 \langle \ell\psi^2\rangle^{\bb{P}^2(\log D)}_{0,\ell} 
\underset{\hbox{\tiny divisor}\atop{\hbox{\tiny equation}}}{=}
 \langle \ell\psi^2\rangle^{\bb{P}^2}_{0,\ell} + 3\underbrace{\langle [\pt]\psi^1\rangle^{\bb{P}^2}_{0,\ell}}_{=1}.$$
\cite{GrP2} and \cite{Over} describe more complicated tropical counts which do give the ordinary invariants for $\bb{P}^2$, but the geometric meaning of their counts was previously mysterious.  The $3\langle [\pt]\psi^1\rangle^{\bb{P}^2}_{0,\ell}$-term above illustrates why some of their tropical curves appear.  We plan to use our result to further illuminate their counts in a paper joint with N. Nabijou.  We encourage the reader to look at the further Examples~\ref{example-double-psi}, \ref{Boundary-Markings} and \ref{ex-hyperell} in \S\ref{sec-applications}.

\subsection{Statement of the main result}\label{MainResult}
Let $X=X_\Sigma$ be a smooth projective toric variety given by a fan $\Sigma$ in $N_\RR:=N\otimes_\bb{Z} \bb{R}$ for some $N\cong \mathbb{Z}^n$. 
A tropical degree $\Delta$ is a map $\Delta:I \rar N$ for a finite index set $I$.  For simplicity, let us assume here that each nonzero $\Delta(i)$ is contained in a ray of $\Sigma$.  Let $D_i$ denote the boundary divisor corresponding to the ray through $\Delta(i)$. Denote $I^{\circ}:=\Delta^{-1}(\{0\})$ and $I^{\partial}:=I\setminus I^{\circ}$.  Curves of degree $\Delta$ have marked points $(x_i)_{i\in I}$ labelled by $I$.  If $i\in I^{\partial}$, then $x_i$ is required to map to $D_i$ with intersection multiplicity equal to the index\footnote{A nonzero element $v\in N$ is called primitive if it is not a positive integer multiple of any other element of $N$.  An element $v$ is said to have index $|v|\in \bb{Z}_{>0}$ if $v$ is equal to $|v|$ times a primitive vector.} of $\Delta(i)$.  For each $i\in I$, let $Z_{A_i}$ denote a generic regularly embedded closed subvariety of $X$ with tropicalization $A_i$.  In particular, if $A_i\subset N_{\bb{R}}$ is an affine-linear subspace with rational slope, then $Z_{A_i}$ is the closure of a general orbit of the torus $\bb{G}(\LL(A_i)\cap N)\subset \bb{G}(N)$.  For $i\in I^{\partial}$, let $Z_{A_i,\Delta(i)}$ denote $Z_{A_i} \cap D_i$.

On the tropical side, corresponding to a degree $\Delta$ curve $C$ (with log structure) is a tropical curve $h:\Gamma \rar N_{\bb{R}}$.  Here, $\Gamma$ is the dual graph to $C$, with marked points $x_i$ of $C$ corresponding to half-edges $E_i$ of $\Gamma$.  The edges of $\Gamma$ are equipped with non-negative integer weights.  Weight $0$ edges are contracted by the map $h$, while the other edges are embedded into lines with rational slope.  Furthermore, $h$ satisfies a ``balancing condition'' and has degree $\Delta$, meaning that for each $i\in I$, $h(E_i)$ points in the direction $\Delta(i)$, and $E_i$ has weight equal to the index of $\Delta(i)$.  The genus of  $\Gamma$ is equal to its first Betti number (plus the sum of the genera of the vertices of $\Gamma$, but these are $0$ for non-superabundant cases).  We say $h:\Gamma\rar N_{\bb{R}}$ satisfies $\A:=(A_i)_{i\in I}$ if $h(E_i)\subset A_i$ for each $i$.  For $\Psi:=(s_i)_{i\in I^{\circ}}$ a tuple of non-negative integers $s_i$, we say $\Gamma$ satisfies $\Psi$ if \begin{align*}
    \val(V)-3\geq \sum_{\substack{i\in I^{\circ} \\E_i\ni V}} s_i
\end{align*} for each vertex $V$.  We say $\Gamma$ is non-superabundant if its vertices all have genus $0$ and if its space of deformations locally has dimension $|I|+(n-3)(1-g)$.

By \cite[Thm. 1.1.2]{ACGS}, a log Gromov-Witten invariant decomposes into types of tropical curves. In particular, given a degeneration, we can decompose into the contribution from superabundant and non-superabundant curves (see Def. \ref{superabundant}), in notation $\langle..\rangle={}^s\langle..\rangle+{}^{\ns}\langle..\rangle$.
It follows from Prop.~\ref{SuperabundantLemma} that ${}^s\langle..\rangle=0$ when $g=0$, or for any $g$ with $n\le 2$ and no $\psi$-conditions, and also in many $n=1$ cases with limited $\psi$-classes (including the Hurwitz counts of \S \ref{section-hurwitz}).  On the other hand, (3) and (4) of Remark \ref{SuperExamples} show that superabundant tropical curves probably appear in most other cases, although it was recently shown \cite[Thm. 3.9]{CJMR2} for stationary invariants of Hirzebruch surfaces  that even though superabundant curves might exists,  ${}^s\langle..\rangle$ still equals $0$.  A general statement about superabundant curves is beyond current knowledge.

\begin{thm} 
\label{mainthmintro}
Let $g\ge 0$ and  $\{s_i\ge 0\}_{i\in I^{\circ}}$ be given and assume that 
$$\sum_{i\in I} \codim A_i +\sum_{i\in I^{\circ}} s_i = |I|+(n-3)(1-g).$$
There exists a log smooth toric degeneration  $X_0(\log D_0)$ of $X(\log D)$ such that the following are non-negative rational numbers that coincide:
\begin{enumerate}
\item The non-superabundant part of the log Gromov-Witten invariant 
$${}^{\ns}\langle  \{[Z_{A_i}]\psi^{s_i}\}_{i\in I^{\circ}}, \{[Z_{A_i,\Delta(i)}]\}_{i\in I^{\partial}}  \rangle^{X_0(\log D_0)}_{g,\Delta},$$
\item The count with multiplicities (cf. \eqref{MultGamma}, or more generally \eqref{GeneralMultGamma}) of rigid marked tropical curves having genus $g$ and degree $\Delta$ satisfying $\A=(A_i)_{i\in I}$ and $\Psi=(s_i)_{i\in I^{\circ}}$.  % with edge $E_i$ mapping to $A_i$ for each $i\in I$, and satisfying the tropical $\psi$-class conditions  as defined in \eqref{PsiTropDfn}.
\end{enumerate}
If ${}^s\langle..\rangle=0$ then these numbers also coincide with $\langle  \{[Z_{A_i}]\psi^{s_i}\}_{i\in I^{\circ}}, \{[Z_{A_i,\Delta(i)}]\}_{i\in I^{\partial}}  \rangle^{X(\log D)}_{g,\Delta}$.
\end{thm}
\begin{proof} 
The equality (1)=(2) is Theorem~\ref{MainThm} and Theorem~\ref{MainThmExtended}.  (2) is easily non-negative rational. 
\end{proof}

A curve in $X$ is called torically transverse if it meets no toric strata of codimension greater than one.

\begin{thm}\label{ThmCount}
In the absence of automorphisms of the tropical curves in (2), (e.g. for $g=0$, see Remark~\ref{rmk-aut}) the count of (1) and (2) furthermore coincides with the non-negative integer-valued count of non-superabundant torically transverse curves in $X_0(\log D_0)$ (or in $X(\log D)$ if ${}^s\langle..\rangle=0$) of genus $g$ and degree $\Delta$ meeting $Z_{A_i}$ at a marked point with a generic $\psi^{s_i}$-condition for each $i\in I^{\circ}$, and meeting $D_i$ in $Z_{A_i}$ with tangency order equal to the index of $\Delta(i)$ for each $i\in I^{\partial}$.  In general, these naive algebraic counts must actually be weighted by $1$ over the numbers of automorphisms of the corresponding marked tropical curves.
\end{thm}

\begin{rmk}
We note that the log curves contributing to (1) in Theorem \ref{mainthmintro} have unobstructed deformations by Proposition \ref{VirtualActual}, so they deform to contribute the same amount to the log Gromov-Witten numbers in the nearby fibers.  In the absence of superabundant curves, this is also what leads to the equality with $\langle  \{[Z_{A_i}]\psi^{s_i}\}_{i\in I^{\circ}}, \{[Z_{A_i,\Delta(i)}]\}_{i\in I^{\partial}}  \rangle^{X(\log D)}_{g,\Delta}$.
\end{rmk}

\begin{myproof}[Proof of Thm. \ref{ThmCount}] 
The count in (1) arises via an intersection cycle $\gamma$ composed (in the case of trivial $\Aut(\Gamma)$) of reduced non-stacky points parametrizing torically transverse curves (Prop.~\ref{PreCount} combined with Lemma~\ref{LogStructureCounts}, plus Lemma~\ref{NoAut} for the non-stackiness). The assertion follows because this cycle is the Gysin restriction of a cycle of stable maps of the total space of the degeneration constructed in \S\ref{sec-tordeg}, and toric transversality is an open condition and stackiness a closed condition.
\end{myproof}

\subsection{Notation}\label{Notation}
Let \gls{kk} be an algebraically closed field of characteristic $0$. Fix an $n\in \bb{Z}_{> 0}$, and let \gls{N} denote a free Abelian group of rank $n$.  Denote $N_{\bb{Q}}:=N\otimes_{\bb{Z}} \bb{Q}$ and $N_{\bb{R}}:=N\otimes_{\bb{Z}} \bb{R}$.  
For a subset $A\subset N_{\bb{Q}}$, \gls{LA} denotes the linear subspace spanned by differences $u-v$ of points in $A$.  In particular, if $A$ is an affine subspace, $\LL(A)$ is the linear subspace parallel to $A$. We write \gls{LNA} for $\LL(A) \cap N$.  Denote $\gls{N'}:=N\oplus \bb{Z}$, $\gls{M}:=\Hom(N,\bb{Z})$, $\gls{M'}:=\Hom(N',\bb{Z})$.  Let $\gls{CA}:=\?{\{(ta,t)|a\in A,t\in \bb{Q}_{\geq 0}\}}\subset N'_{\bb{Q}}$ denote the closed cone over $A$.  In particular, $\gls{LCA}:=\LL(\C(A))$ denotes the linear closure of $A\times \{1\}$ in $N_{\bb{Q}}\times \bb{Q}$.
For a $\ZZ$-module $V$, we set $\GG(V)=\GG_m(\kk)\otimes_\ZZ V$.

For any scheme $Y$ we will write $Y^{\dagger}$ to denote $Y$ equipped with a log structure.  Similarly, a morphism of log schemes will be given a $\dagger$ in its superscript. 
We will sometimes use Witten's correlator notation for Gromov-Witten invariants $\langle insertions\rangle^{X}_{g,\beta}$ and similarly for log Gromov-Witten invariants 
$\langle insertions\rangle^{X(\log D)}_{g,\Delta}$
for the divisorial log structure of a divisor $D$ on $X$. 
For the latter, note that the refined degree $\Delta$ contains information about the orders of tangencies and numbers of markings that go to each component of $D$.

For the reader's convenience, we will keep track of important notation from \S \ref{TropCurves}-\S\ref{Sing} in the glossary below.
\setlength\columnseprule{.6pt}
\begin{multicols}{2}
\printglossary[title=Notation glossary] 
\end{multicols}

\subsection{Acknowledgements}
The authors are grateful to Mark Gross for suggesting the Nishinou-Siebert approach to descendant Gromov-Witten invariants. We thank Bernd Siebert for encouraging the use of log GW theory and answering questions.  We also thank Y.P. Lee in this context.  We thank Tony Yue Yu, Rahul Pandharipande and Bernd Siebert for inspiring us to view $\psi$-classes as being pulled back from the underlying moduli space of stable curves rather than our early approach of using tangency conditions.  We thank Peter Overholser for his repeated warnings about negative counts and useful examples he gave. We thank Mark Gross, Andreas Gross, Hannah Markwig and Renzo Cavalieri for explaining their related projects. David Rydh gave a useful answer to a question on Chow groups for stacks. We thank the anonymous referee for suggesting various improvements.

\section{Tropical curves}\label{TropCurves}
We follow \cite{Mi,NS}. While definitions in this section might appear ad hoc, we will later see in \S \ref{tropicalization} that the herein defined objects are natural because they are the output of \emph{tropicalization}. 

Let $\gls{Gammabar}$ denote the topological realization of a finite connected graph.  Let $\gls{Gamma}$ be the complement of some subset of the $1$-valent vertices of $\?{\Gamma}$.  Let $\gls{Gamma0}$, $\gls{Gamma1}$, $\gls{Gamma1inf}$, and $\gls{Gamma1c}$ denote the the sets of vertices, edges, non-compact edges, and compact edges of $\Gamma$, respectively.  We equip $\Gamma$ with a ``weight-function'' $\gls{w}:\Gamma^{[1]}\rar \bb{Z}_{\geq 0}$ and a ``genus-function'' $\gls{g}:\Gamma^{[0]} \rar \bb{Z}_{\geq 0}$ (typically writing $g(V)$ as $g_V$), and we require that univalent and bivalent vertices have positive genus.  %We let $m=\#w^{-1}(0)$ be the number of weight $0$ edges, and let $e_{\infty}:=\# \Gamma^{[1]}_{\infty}-m$ be the number of positive weight unbounded edges. 
A {\bf marking} of $\Gamma$ is a bijection $\gls{eps}:I\rar \Gamma^{[1]}_{\infty}$ for some index set $I$.

Denote $E_i:=\epsilon(i)$.  We let $\gls{Icirc}:=\epsilon^{-1}(w^{-1}(0))\subset I$, i.e., the set of $i\in I$ for which $w(E_i)=0$, and we call these the \textbf{interior markings}. 
The complement $\gls{Ip} = I\setminus I^{\circ}$ forms the \textbf{boundary markings}.  Denote by $\gls{Gammaeps}$ the data of $\Gamma$, the weight-function $w$ and genus-function $g$ (left out of the notation), and the marking.  Let $\gls{einf}:=\# I$.  For each vertex $V$, we let $\gls{IVcirc}$ denote the set of $i\in I^{\circ}$ such that $V\in E_i$, and let $\gls{mV}:=\# I_V^{\circ}$.  We call $V$ \textbf{marked} if $m_V>0$ and unmarked if $m_V=0$.

\begin{dfn}\label{TropCurveDfn}
A {\bf parametrized marked tropical curve} \gls{ptc} is data $(\Gamma,\epsilon)$ as above, along with a continuous map $\gls{h}:\Gamma\rar N_{\bb{R}}$ such that 
\begin{enumerate}
\item For each edge $E\in \Gamma^{[1]}$, $h|_E$ is constant if and only if $w(E)=0$.  Otherwise,
$h|_E$ is a proper embedding into an affine line with rational slope.  We will refer to edges $E$ with $w(E)=0$ as {\bf contracted edges}.  In particular, {\bf self-adjacent} edges (those with the same vertex at both ends) must have weight zero as there is no way to embed a loop into a line.
\item For every $V\in \Gamma^{[0]}$, $h(V) \in N_{\bb{Q}}$, and the following {\bf balancing condition} holds.  For any edge $E\ni V$, denote by $\gls{uVE}$ the primitive integral vector emanating from $h(V)$ into $h(E)$ (or $u_{(V,E)}:=0$ if $h(E)$ is a point). Then
\begin{align*}
\sum_{E\ni V} w(E) u_{(V,E)} =0.
\end{align*}
\end{enumerate}
Furthermore, for each contracted edge $E$, we have the additional data of a length $\ell(E)\in \bb{Q}_{>0}$. 
For the purposes of this paper, we assume that $h$ is non-constant, or equivalently, that $e_\infty>0$. Also, to simplify the exposition, we assume throughout that $\Gamma^{[0]}\neq \emptyset$, and we handle this case separately in Remark \ref{NoVertices}.

For non-compact edges $E_i \ni V$, we may denote $u_{(V,E_i)}$ simply as $u_{E_i}$ or $\gls{ui}$ (with $u_i=0$ if $E_i$ is contracted).  
Similarly, for any edge $E$, we may simply write $\gls{uE}$ when the vertex is either clear from context or unimportant, e.g., as in $\bb{Z}u_{E}$.

An {\bf isomorphism} of marked parametrized tropical curves $(\Gamma,\epsilon,h)$ and $(\Gamma',\epsilon',h')$ is a homeomorphism $\Phi:\Gamma\rar \Gamma'$ respecting the weights, markings, and lengths such that $h=h'\circ \Phi$.  A {\bf marked tropical curve} (or, for short, a {\bf tropical curve}) is then defined to be an isomorphism class of parametrized marked tropical curves.  We will use $(\Gamma,\epsilon,h)$ to denote the isomorphism class it represents and will often abbreviate this as simply $h$ or $\Gamma$. We will denote the automorphism group of the tropical curve represented by $(\Gamma,\epsilon,h)$ as $\gls{AutG}$.

\begin{rmk} \label{rmk-aut}
Since unbounded edges are marked, we note that $\Aut(\Gamma)$ is always trivial for $g=0$. More generally, if no two distinct vertices map to the same point, then automorphisms can only act by permuting edges
that have the same vertices, and by Lemma \ref{lem-non-span-loops}, curves with such collections of edges are superabundant unless $n=1$.  Example \ref{ex-hyperell} features nontrival $\Aut(\Gamma)$.  Note that after the removal of boundary markings, automorphisms may occur also for $g=0$ and any $n$,  cf. \S\ref{OrdinaryGW}.
\end{rmk}

A {\bf tropical immersion} is a tropical curve with no higher-genus vertices, no self-adjacent edges, and no pairs of edges which share a vertex and point in the same direction.

If $b_1(\Gamma)$ denotes the first Betti number of $\Gamma$, the {\bf genus} of a tropical curve $\Gamma$ is defined as
$$
\gls{gGamma}= b_1(\Gamma)+\sum_{V\in\Gamma^{[0]}}g_V.
$$
\end{dfn}

Set $\gls{ebar}:=\# \Gamma^{[1]}_c$, so $\bar e+e_\infty$ is the total number of edges of $\Gamma$.  
Let $\gls{valV}$ denote the valence of a vertex $V$ (counting self-adjacent edges twice).  Motivated by considering a valency of three to be generic,\footnote{It will sometimes be useful to think of $\ov(V)$ geometrically as the dimension of $\?{\s{M}}_{0,\val(V)}$.} one defines the over-valence $\gls{ovV}:=\val(V)-3$, and \[\gls{ovG}:=\sum_{V\in \Gamma^{[0]}} \ov(V).\]  

Now assume $g_V=0$ for every $V\in \Gamma^{[0]}$.  Note that $\ov(V)\ge 0$ for all $V\in \Gamma^{[0]}$.
Let $\gls{flags}$ denote the set of {\bf flags} $(V,E)$, $V\in E$, of $\Gamma$ (with self-adjacent edges contributing twice). Computing $\#F\Gamma$ once via vertices, once via edges, yields
\begin{align}\label{v-e-equality}
3\#\Gamma^{[0]}+\ov(\Gamma) =  e_{\infty}+2\?{e}.
\end{align}
A computation of the Euler characteristic of $\Gamma$ yields $1-g = \#\Gamma^{[0]} - \?{e}$. Using this to eliminate $\#\Gamma^{[0]}$ in \eqref{v-e-equality} gives
\begin{align}\label{e-infty}
e_{\infty}=\?{e}-3g+3+\ov(\Gamma).
\end{align}

The {\bf type} of a marked tropical curve (possibly with positive $g_V$'s) is the data  $(\Gamma,\epsilon)$ (including the weights and $g_V$'s and up to homeomorphism respecting this data and the markings), along with the data of the map $u:F\Gamma\rar N$, $(V,E)\mapsto u_{(V,E)}$.  Let $\gls{typespace}$ denote the space of marked tropical curves of type $(\Gamma,\epsilon,u)$.  We may write $\f{T}_{\Gamma}$ for short.  If $\Gamma$ is genus $0$, we have
\begin{align}\label{Tshape}
\f{T}_{(\Gamma,\epsilon,u)}\cong N_{\bb{Q}} \times \bb{Q}_{>0}^{\?{e}}.
\end{align}
  Indeed, we can specify $h\in \f{T}_{(\Gamma,\epsilon,u)}$ by specifying the image in $N_{\bb{Q}}$ of some $V\in \Gamma^{[0]}$ and then specifying the lengths of the compact edges. With a similar approach, we obtain more generally for curves of any genus that
\begin{align}\label{Tshape-general}
\f{T}_{(\Gamma,\epsilon,u)}\cong N_{\bb{Q}} \times (\bb{Q}_{>0}^{\?{e}}\cap W),
\end{align}
where $W$ is the subspace of $\bb{Q}^{\?{e}}$ generated by vectors of the form  $(a_{E})_{E\in \Gamma^{[1]}_c}$ such that
for all closed loops in $\Gamma$, if $L$ is the set of edges of such a loop, we have
$$0=\sum_{E\in L} a_E u_E.$$
Here, a {\bf loop} in $\Gamma$ is defined to be a cyclically ordered sequence of edges $(E_1,...,E_r)$ in $\Gamma$, equipped with orientations with respect to which the starting point of $E_{i+1}$ is the ending point of $E_{i}$ for each $i=1,\ldots,r$.  The directions of the $u_E$'s above are chosen to respect this orientation.  $W$ is thus the constraint to have the loops close up.  As we will see, the dimension of $W$ and hence of $\f{T}_{(\Gamma,\epsilon,u)}$ can vary even for fixed $n$, $\?{e}$, and $g$.  The maximal codimension $W$ can have is $gn$, so \eqref{Tshape-general} motivates the following definition.
\begin{dfn} \label{superabundant}
A tropical curve $\Gamma$ is called {\bf superabundant} if it contains at least one vertex of positive genus,\footnote{Superabundance of a tropical curve should mean that the deformations of the corresponding log curves are obstructed.  In this sense, it is perhaps not appropriate to always say tropical curves with higher-genus vertices are superabundant.  However, Proposition \ref{No-gV} indicates that it is at least appropriate to call these curves superabundant whenever they also satisfy a generic and general collection of conditions.} or if every vertex is genus $0$ but the inequality $\dim \f{T}_{(\Gamma,\epsilon,u)}\ge n+\?{e}-ng$ is strict (i.e. not an equality).
\end{dfn}

Assume again that $g_V=0$ for each $V$.  The span of a loop is defined to be the span of the directions of its edges.  One easily sees the following sufficient condition for superabundance:
\begin{lem}\label{lem-non-span-loops}
If $\Gamma$ has a loop that does not span $N_\QQ$, then $\Gamma$ is superabundant.  In particular, non-superabundant curves have no self-adjacent edges.
\end{lem}

See Fig.~\ref{trop-superabundant} for an example demonstrating that the converse of Lemma~\ref{lem-non-span-loops} is false.  We do, however, have the following partial converses:
\begin{prop}\label{SuperabundantLemma}  
For g=0, all tropical curves are non-superabundant.
If $n=1$, all tropical curves without self-adjacent edges are non-superabundant.
For $n=2$, all tropical immersions for which $\val(V)>3$ holds at most at one vertex are non-superabundant.
\end{prop}
\begin{proof}
All the statements are easily checked except the last one for $n=2$. This is essentially \cite[Prop. 2.23]{Mi}, noting that the vertex with $V$ with $\val(V)>3$ can be chosen as the start for Mikhalkin's inductive construction of a maximal ordered tree in the curve.
\end{proof}

\begin{rmk}\label{SuperExamples} 
\begin{enumerate}
\item A typical example for a superabundant $g=1$, $n=2$ curve with no contracted edges is the following curve supported on a line (a sequence of a weight $2$ edge, double edge, weight $2$ edge):    
\xymatrix{
\ar@{-}^2[r]^(-.15){}="a"^(1.15){}="b" \ar@{-} "a";"b"
&\bullet
\ar@{=}[r]^(-.1){}="a"^(1.1){}="b" \ar@{=} "a";"b"
&\bullet
\ar@{-}^2[r]^(-.15){}="a"^(1.15){}="b" \ar@{-} "a";"b"
&}

\item 
Mikhalkin provides an example for a superabundant tropical immersion for $n=2, g=22$, \cite[Rem. 2.25]{Mi} based on Pappus theorem. An $n=2$, $g=3$ superabundant immersion is given in \cite[Example~3.10]{GM07}.
\item Superabundant curves of any higher genus can be generated from genus zero curves by attaching self-adjacent edges (or similarly, by replacing an unmarked trivalent vertex by a triangle of edges, or by inserting curves like from (1) above into edges with weight greater than $1$).  Since we generally want to avoid superabundancy, this observation is particularly problematic if $n\ge 3$ because then the expected dimension $d^{\trop}_{g,\Delta}$ (see \eqref{GeneralDimension}) increases or (for $n=3$) stays the same as we decrease $g$.  Hence, given some fixed constraints to be met for some degree (constraints and degrees will be introduced below), it should be at least as easy to find curves satisfying these constraints in genus zero as in higher genus, and then from such curves we can attach self-adjacent edges (or apply the other modifications we just described) to increase the genus.  One therefore expects that for $n\geq 3$ and $g>0$, superabundant curves are abundant.
\item For $n=2$, increasing $g$ by $1$ only increases the expected dimension $d^{\trop}_{g,\Delta}$ by $1$.  So if one considers a rigid collection of conditions for given $g>0$ and $\Delta$ and then forgets a $\psi$-class (i.e., decreases one of the $s_i$'s in \eqref{PsiTropDfn} below by $1$), one would expect there to be a $0$-dimensional collection of curves of genus $g-1$ which satisfy the modified conditions.  Given such a genus $g-1$ curve, one could then insert a self-adjacent edge at the vertex where the psi-class was forgotten to get a superabundant curve of the desired genus which satisfies the original collection of conditions (these are not tropical immersions, so Prop. \ref{SuperabundantLemma} does not apply).  So even in dimension $2$, it seems that higher-genus counts with $\psi$-class insertions will typically involve the appearance of superabundant curves, as previously noted in \cite[Appx.~B]{Bou}.  However, in at least some cases, the contributions of the superabundant curves is trivial, e.g., for stationary invariants of Hirzebruch surfaces according to  \cite[Thm. 3.9]{CJMR2}.
\end{enumerate}
\end{rmk}

\begin{figure}
\resizebox{0.4\textwidth}{!}{
\includegraphics{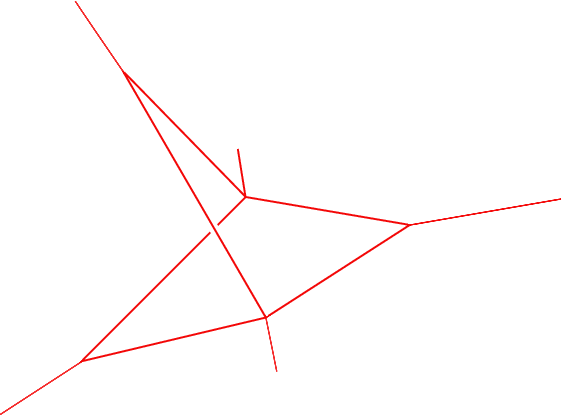}
}
\caption{A tropical genus two curve in $\RR^3$ that is superabundant even though all its loops span.}
\label{trop-superabundant}
\end{figure}

Allow $g_V>0$ again.  The {\bf degree} $\gls{Delta}$ of a type $(\Gamma,\epsilon,u)$ is the data of the index set $I$, along with the map
\begin{align}\label{Delta}
\Delta:I\rar N, \hspace{.5 in} \Delta(i):=w(E_i)u_i.
\end{align}
Note that the data of $I$ (hence of $e_{\infty}$) is part of the data of the degree $\Delta$ but is suppressed in the notation.

\begin{lem}[\cite{NS}, Proposition 2.1]\label{Types} The number of types of tropical curves of fixed degree, fixed number of markings, and fixed genus is finite.\footnote{Actually, \cite{NS} does not allow for contracted edges or higher-genus vertices like we do, but since the maximum possible number of such vertices and edges is bounded (chains of such edges eat up genus or markings), their Proposition 2.1 is easily generalized to our setup.}
\end{lem}

\subsection{Incidence and $\psi$-class conditions}
Here we consider only ``affine constraints.''  See \S \ref{GenInc} for a discussion of how to generalize to other tropical constraints.
\begin{dfn}\label{Constraints}
An {\bf affine constraint} $\gls{AA}$ for the degree $\Delta$ is a tuple $(A_i)_{i\in I}$ of affine subspaces of $N_{\bb{Q}}$ such that $\Delta(i)\in \LL(A_i)$ for each $i$.  A marked tropical curve $(\Gamma,\epsilon,h)$ {\bf matches the constraints $\A$} if $h(E_i)\subset A_i$ for each $i\in I$.

Consider another tuple $\gls{Psi}:=(s_i)_{i\in I^{\circ}} \in \bb{Z}_{\geq 0}^{|I^{\circ}|}$.  
We say $(\Gamma,\epsilon,h)$ satisfies $\Psi$ if for each vertex $V$ with $g_V=0$, we have\footnote{Based on the proof of Proposition \ref{Log2Trop}, we expect the appropriate $\psi$-class condition for higher genus vertices to be $(3-n)g_V+\ov(V) \geq \sum_{i\in I_V^{\circ}} s_i$.  However, for our purposes we can say that $\Psi$ imposes no conditions on higher-genus vertices since the non-superabundance assumption will still prevent curves with such vertices from contributing.}
\begin{align}\label{PsiTropDfn}
    \ov(V) \geq \sum_{i\in I_V^{\circ}} s_i.
\end{align}

We are interested in the space
\[\gls{TAPsi}\]
of marked tropical curves of genus $g$ and degree $\Delta$, matching the constraints $\A$ and satisfying the $\psi$-class conditions $\Psi$.

For a marked vertex $V \in \Gamma^{[0]}$ such that $g_V=0$ and \eqref{PsiTropDfn} is an equality, let $I_V^{\circ} = \{i_1,\ldots,i_{m_V}\}$, and then let $\gls{<V>}$ denote the multinomial coefficient
\begin{align}\label{V}
\langle V \rangle :=\binom{\ov(V)}{s_{i_1},\ldots,s_{i_{m_V}}} := \frac{\ov(V)!}{\prod_{i\in I_V^{\circ}} s_i!}.
\end{align}
 For unmarked $V$, $\langle V\rangle:=1$.  We won't need an analog for higher-genus vertices in this paper.
\end{dfn}

The factor $\langle V \rangle$ will contribute to the multiplicity with which we will count tropical curves.  
Example~\ref{example-double-psi} demonstrates this.
Note that $\langle V\rangle = 1$ whenever $m_V=1$ and \eqref{PsiTropDfn} is an equality, and consequently, $\langle V\rangle = 1$ holds generically for all vertices when there are no $\psi$-classes.  
As we will see, the reason for this multiplicity $\langle V \rangle$ is that it is equal to $\int_{[\?{\s{M}}_{0,\val(V)}]} \prod_{i\in I_V^{\circ}} \psi_{i}^{s_i}$ (cf. \cite[Lemma 1.5.1]{Kock}). 

Recall that for any affine subspace $A$, we write $\LL(A)$ for the linear subspace parallel to $A$.  The following proposition describes the subspaces of $\f{T}_{g,\Delta}(\A,\Psi)$ corresponding to curves of a fixed type. We first state an elementary lemma that it uses.

\begin{lem} \label{affine_translates}
Given two affine subspaces $A_1,A_2\subset N_\bb{Q}$, if $A'_1$ is a general translate of $A_1$ then the intersection $A'_1\cap A_2$ is either empty or transverse (i.e. $\codim A'_1\cap A_2=\codim A_1+\codim A_2$). 
\end{lem}

\begin{prop}\label{TropDeform}
Let $(\Gamma,\epsilon,h)\in \f{T}_{g,\Delta}(\A,\Psi)$.  For each edge $E\in \Gamma^{[1]}_c$, fix a labeling of its two vertices as $\partial^+E$ and $\partial^-E$.  For $E\in \Gamma^{[1]}_{\infty}$, define $\partial E$ to be the unique vertex contained in $E$.  Define a map of vector spaces 
\begin{align}\label{D0}
\gls{PhiQ}:\prod_{V\in \Gamma^{[0]}} N_{\bb{Q}} &\rar \left(\prod_{E\in \Gamma^{[1]}_c} N_{\bb{Q}}/\bb{Q}u_{E}\right) \times \left( \prod_{i\in I} N_{\bb{Q}}/\LL(A_i) \right)  \\
H &\mapsto ((H_{\partial^+E}-H_{\partial^-E})_{E\in \Gamma^{[1]}_c},(H_{\partial E_i})_{i\in I}) \notag
\end{align}
Then the space of tropical curves in $\f{T}_{g,\Delta}(\A,\Psi)$ of the same combinatorial type $u$ as $(\Gamma,\epsilon,h)$ can be identified with a non-empty open convex polyhedron (i.e., an intersection of finitely many open affine half-spaces in $\ker(\Phi_{\bb{Q}})$.  Furthermore, assuming that the constraints are in general position (in the space of possible translations of the constraints), and that $g_V=0$ for each $V$, the dimension of this space is
\[\dim[\ker(\Phi_{\bb{Q}})] = \dim \f{T}_{(\Gamma,\epsilon,u)} - \sum_{i\in I} \codim(A_i). \]
If $\Gamma$ is not superabundant, this becomes $$\dim[\ker(\Phi_{\bb{Q}})] = n+\?{e} -ng -\sum_{i\in I} \codim(A_i).$$   
\end{prop}
\begin{proof}
Points $H=(H_V)_V$ on the left-hand side of (\ref{D0}) give deformations of $h$ by adding $H_V$ to $h(V)\in N_{\bb{Q}}$.  For $E\in \Gamma^{[1]}_c$, $h(E)$ is deformed to be the line segment connecting $h(\partial^{\pm}E) + H_{\partial^{\pm}E}$.  The image of an unbounded edge $E$ is deformed to $(h(E)+H_{\partial E})$.  By construction, these are the deformations of $h$ which still embed edges into affine lines of rational slope or contract them, but the balancing conditions and constraints may not hold.  The deformed curve corresponding to some $H$ will have the original combinatorial type (and in particular be balanced) if and only if $H$ is in the kernel of the first factor of $\Phi_{\bb{Q}}$ and, for each $E\in \Gamma^{[1]}_c$, the deformed edge vector $((h(\partial^+E)+H_{\partial^+E})-(h(\partial^-E)+H_{\partial^-E}))$ is a positive multiple of the original one
$h(\partial^+E)-h(\partial^-E)$ (this is an open affine half-space condition).  Similarly, it will satisfy the incidence conditions if and only if $H$ is in the kernel of the second factor.  This proves the first claim. 

We prove the second claim by induction.  If there are no incidence conditions (i.e., each $A_i$ equals $N_{\bb{Q}}$) then the claim is trivial.  Let $P$ be a small generic translate of $\ker(\Phi_{\bb{Q}})$  corresponding to the subspace of deformations of $h$ which satisfy all the conditions we have imposed so far (with the open half-space conditions removed since we are only interested in dimensions right now).  Suppose we add a constraint $A_i+\epsilon$, where $\epsilon$ is a small generic element of $N_{\bb{Q}}$.  This adds a factor $N_{\bb{Q}}/\LL(A_i)$ to the right-hand side, and an element $H$ in the domain of $\Phi_{\bb{Q}}$ maps to $0$ in this factor if and only if $H_i\in \LL(A_i)+\epsilon$.  This corresponds to intersecting $P$ with $(\LL(A_i)+\epsilon)$.  
By Lemma~\ref{affine_translates}, these intersections are transverse, so this proves the claim.

Finally, the claim for the non-superabundant case follows from Def.~\ref{superabundant}.
\end{proof}

\begin{lem} \label{phi-surjective}
For $\A$ generic and $(\Gamma,\epsilon,h)\in \f{T}_{g,\Delta}(\A,\Psi)$ non-superabundant, $\Phi_{\bb{Q}}$ is surjective.
\end{lem}
\begin{proof} Prop.~\ref{TropDeform} gives the rank of the kernel, so to show that the rank of the cokernel of $\Phi_{\bb{Q}}$ vanishes we may compute
$\dim\coker(\Phi_{\bb{Q}})=\dim\ker\Phi_{\bb{Q}}+\dim\codomain(\Phi_{\bb{Q}})-\dim \domain\Phi_{\bb{Q}}$ giving
\begin{align*} 
\left(n+\?{e} -ng -\sum_i \codim(A_i) \right) + \left((n-1)\?{e}  +\sum_i \codim(A_i) \right) - n\#\Gamma^{[0]}
\end{align*}
which is seen to vanish by inserting $\#\Gamma^{[0]}= \?{e}+1-g$.
\end{proof}

Note that for non-superabundant $(\Gamma,\epsilon,h)\in \f{T}_{g,\Delta}(\A,\Psi)$ with each $g_V=0$, it follows from \eqref{PsiTropDfn} that 
\begin{align} \label{ov-inequality}
\ov(\Gamma) \ge \sum_{i\in I^{\circ}} s_i.
\end{align}

\begin{dfn}\label{GeneralDfn}
Fix a genus $g$, a degree $\Delta$, and a tuple $\Psi$ as above.  
An affine constraint $\A$ of codimension ${\bf a}$ is {\bf general} for $\Delta$ and $\Psi$ if 
\begin{align}\label{GeneralDimension}
\sum_{i\in I^{\circ}} s_i  + \sum_{i\in I} \codim(A_i) = \gls{dtrop}:= e_{\infty}+(n-3)(1-g),
\end{align}
and for any non-superabundant $(\Gamma,\epsilon,h)\in \f{T}_{g,\Delta}(\A,\Psi)$ with each $g_V=0$, \eqref{ov-inequality} is an equality and $\Gamma$ has no contracted compact edges. 
\end{dfn}

\begin{rmk}\label{dtroprmk}
One should think of $d^{\trop}_{g,\Delta}$ as the expected dimension of the moduli space of genus $g$ degree $\Delta$ tropical curves in $N_{\bb{Q}}$.  Indeed, by \eqref{e-infty}, 
\begin{align}\label{dtrop}
d^{\trop}_{g,\Delta}= n+\ov(\Gamma)+\?{e}-ng
\end{align}  for any such tropical curve $\Gamma$.  From Lemma \ref{General} below, generic non-superabundant $\Gamma$ should be general, which in the absence of $\psi$-class conditions means that all vertices are trivalent (thus $\ov(\Gamma)=0$). % $\partial E_i\neq \partial E_j$ for distinct $i,j\in I^{\circ}$.
 In this case, $d^{\trop}_{g,\Delta}$ reduces to $n+\?{e}-ng$, which by Definition \ref{superabundant} is the dimension of the tropical moduli space in a neighborhood of the non-superabundant $\Gamma$.
\end{rmk}

\begin{lem}\label{General}
Fix $\Delta$ and $\Psi$.  
Generic $\A$ satisfying (\ref{GeneralDimension}) are general, and furthermore, in this case, $\f{T}_{g,\Delta}(\A,\Psi)$ has at most one tropical curve of any given non-superabundant combinatorial type, hence the set of non-superabundant tropical curves is finite.
\end{lem}
\begin{proof}
Let $(\Gamma,\epsilon,h)\in \f{T}_{g,\Delta}(\A,\Psi)$ be non-superabundant with each $g_V=0$.  Assume for now that $\Gamma$ has no contracted compact edges.  Substituting \eqref{dtrop} into \eqref{GeneralDimension} and rearranging, we see that \eqref{GeneralDimension} is equivalent to
\begin{align}\label{GenDim}
\left(\sum_{i\in I^{\circ}} s_i \right) - \ov(\Gamma) = n+\?{e}-ng  - \sum_{i\in I} \codim(A_i).
\end{align}
Now, if $\A$ is generic and satisfies \eqref{GenDim}, Proposition \ref{TropDeform} tells us that the subspace of $\f{T}_{g,\Delta}(\A,\Psi)$ of curves with the same type as $(\Gamma,\epsilon,h)$ is a convex polyhedron of dimension $d:=\left(\sum s_i\right)-\ov(\Gamma)$.  This dimension is non-negative since there exists the given $(\Gamma,\epsilon,h)\in \f{T}_{g,\Delta}(\A,\Psi)$, so \eqref{ov-inequality} must be an equality, as desired. 

If $\Gamma$ did have contracted compact edges, none of these were (or would combine to) loops since $\Gamma$ is non-superabundant. We could imagine these edges were contracted before applying $h$, resulting in a modification $\Gamma'$ of $\Gamma$ without contracted edges.  Then we can apply the above reasoning to show that \eqref{ov-inequality} is an equality for $\Gamma'$. But since $\ov(\Gamma')>\ov(\Gamma)$, \eqref{ov-inequality} cannot hold for $\Gamma$.  Hence, non-superabundant curves with edge contractions do not occur generically, and so we proved the first assertion.

Since $d$ from above is $0$, the uniqueness of the combinatorial type follows since a nonempty $0$-dimensional convex polyhedron must be a single point.  The finiteness of non-superabundant curves in $\f{T}_{g,\Delta}(\A,\Psi)$ now follows from Lemma \ref{Types}.
\end{proof}

Let $\gls{Tns}\subset \f{T}_{g,\Delta}(\A,\Psi)$ denote the subset consisting of non-superabundant tropical curves.  If $\A$ is generic and general for $\Delta$ and $\Psi$, we will call $\Gamma\in \f{T}^{\ns}_{g,\Delta}(\A,\Psi)$ {\bf rigid} and we also then call $\f{T}^{\ns}_{g,\Delta}(\A,\Psi)$ itself rigid.
%%\footnote{In light of Remark \ref{SuperExamples}(3), we generally expect superabundant curves in $\f{T}_{g,\Delta}(\A,\Psi)$ when $g\geq 1$ and $n\geq 3$. We will show in Lemmas \ref{Log2Trop} and \ref{No-gV} that if $\f{T}^{\ns}_{g,\Delta}(\A,\Psi)$ is rigid, then it includes the non-superabundant tropicalizations of all log curves satisfying the corresponding algebraic conditions if $g_V=0$ for all $V$.}

In \S \ref{OrdinaryGW}, we will want to use the following:
\begin{lem}\label{VerticesDistinct}
Suppose $\Gamma$ is in a rigid $\f{T}^{\ns}_{g,\Delta}(\A,\Psi)$ with $g=0$.  Then no two vertices of $\Gamma$ map to the same point of $N_{\bb{Q}}$.
\end{lem}
\begin{proof}
If two distinct vertices do map to the same point of $N_{\bb{Q}}$, we can identify them to form a genus $1$ tropical curve $\Gamma'$ satisfying the same incidence conditions as $\Gamma$.  But since there are by genericity no contracted compact edges, the unique loop of $\Gamma'$ spans a positive dimensional subspace of $N_{\bb{Q}}$, hence decreases the dimension of $\f{T}_{\Gamma}$.  But then by Proposition \ref{TropDeform}, the dimension of the space of curves of the same type as $\Gamma'$ satisfying $\A$ has dimension $<0$, contradicting that it contains $\Gamma'$.
\end{proof}

Lemma \ref{General} implies that a rigid $\f{T}^{\ns}_{g,\Delta}(\A,\Psi)$ is finite, so we can hope to count its elements.  We now describe the multiplicity we will use to perform this count.  It is a generalization of the Nishinou-Siebert multiplicity \cite[Prop. 5.7]{NS}. See \cite{MRudMult} for alternative formulas for $\Mult(\Gamma)$ which are more practical for some applications.  Recall that we denote $\LL_N(A):=\LL(A)\cap N$.

\begin{lemdfn}\label{Ddfn}
Fix some $(\Gamma,\epsilon,h)$ in a rigid $\f{T}^{\ns}_{g,m,\Delta}(\A,\Psi)$.  The map 
\begin{align}\label{D}
\gls{Phi}:=\prod_{V\in \Gamma^{[0]}} N &\rar \left(\prod_{E\in \Gamma^{[1]}_c} N/\bb{Z}u_{\partial^-E,E}\right) \times \left(\prod_{i\in I} N/\LL_N(A_i) \right) \\
H &\mapsto ((H_{\partial^+E}-H_{\partial^-E})_{E\in \Gamma^{[1]}_c},(H_{\partial E_i})_{i\in I}) \notag
\end{align}
is an inclusion of lattices of finite index. 

Recall the definition of $\langle V\rangle$ from \eqref{V}.  We denote 
\begin{align}\label{DGamma}
   \gls{DGamma}:= \inde(\Phi) \prod_{V\in \Gamma^{[0]}} \langle V \rangle,
\end{align}
and we define 
\begin{align}\label{MultGamma}
\gls{Mult}:=\frac{\f{D}_{\Gamma}}{|\Aut(\Gamma)|} \prod_{E\in \Gamma^{[1]}_c} w(E).
\end{align}

\end{lemdfn}
\begin{proof}
Tensoring $\Phi$ with $\bb{Q}$ gives the map $\Phi_{\bb{Q}}$ from Proposition \ref{TropDeform}.  
Thus, rigidity implies that $\ker(\Phi)$ is trivial, and by Lemma~\ref{phi-surjective}, $\Phi_{\bb{Q}}$ is surjective, so $\Phi$ is indeed injective and finite index.
\end{proof}

We can now define our tropical counts:

\begin{dfn} For $\f{T}^{\ns}_{g,\Delta}(\A,\Psi)$ rigid, 
\begin{align*}
\gls{GWtrop}:=\sum_{(\Gamma,\epsilon,h)\in \f{T}^{\ns}_{g,\Delta}(\A,\Psi)} \Mult(\Gamma).
\end{align*}
If $\f{T}_{g,\Delta}(\A,\Psi)=\f{T}^{\ns}_{g,\Delta}(\A,\Psi)$, we may write this as simply $\GW_{g,N_{\bb{Q}},\Delta}^{\trop}(\A,\Psi)$.
\end{dfn}

\section{Descendant log Gromov-Witten invariants}

\subsection{Toric degenerations} \label{sec-tordeg}
 As in \S 3 of \cite{NS} (to which we refer for more details), we can produce from a polyhedral decomposition $\gls{P}$ of $N_{\bb{Q}}$ a toric degeneration $X\rar \bb{A}^1$ as follows.  Consider the lattice $N':=N\times \bb{Z}$.  We can embed $\s{P}$ into $N'_{\bb{Q}}$ by identifying $N_{\bb{Q}}$ with $(N_{\bb{Q}},1)\subset N'_{\bb{Q}}$, and we then produce a fan $\gls{SigmaPprime}$ by taking the cone $C(\s{P})$ over $\s{P}$.  That is, a $d$-dimensional cell $\Xi\in \s{P}^{[d]}$ yields a $(d+1)$-dimensional cone $C(\Xi)\in \Sigma'_{\s{P}}$ defined as in \S \ref{Notation}.
 Then $\gls{sX}$ is defined to be the toric variety $X(\Sigma'_{\s{P}})$ produced from the fan $\Sigma'_{\s{P}}$.

Recall that the fan $\Sigma_{\bb{A}^1}$ in $\bb{Q}$ with cones $\{0\}$ and $\bb{Q}_{\geq 0}$ corresponds to the toric variety $\bb{A}^1$.  The projection $N_{\bb{Q}}\times \bb{Q} \rar \bb{Q}$ induces a map of fans $\Sigma'_{\s{P}}\rar \Sigma_{\bb{A}^1}$, hence a map $f:\s{X} \rar \bb{A}^1$.   
We may assume that $f$ is projective by requiring the subdivision $\s{P}$ to be induced by a piecewise affine convex function.

Let $\gls{Xt}$ denote the fiber of $f$ over a closed point $t\in \bb{A}^1$.  For $t\neq 0$, $\s{X}_t$ is torically isomorphic to $X:=X(\Sigma_{\s{P}})$, where $\gls{SigmaP}=\{\lim_{\lambda\ra 0}\lambda\Xi\,\mid\,\Xi\in\s{P}\}$ denotes the asymptotic fan of $\s{P}$.  For a ray $\rho \in \Sigma_{\s{P}}$ generated by a vector $u$, we let $D_{\rho}$ or $D_u$ denote the corresponding divisor of $\s{X}$.

We now describe $\s{X}_0$.  For each pair of cells $\Xi\subseteq \Xi'$ of $\s{P}$, the rays from $\Xi$ through $\Xi'$ define a cone $\sigma_{\Xi,\Xi'} \subset N_{\bb{Q}}/\LL(\Xi)$ ($\LL(\Xi)$ denoting the linear subspace generated by vectors in $\Xi$).  The union over all $\Xi'$ containing a fixed $\Xi$ forms a complete fan $\Sigma_{\Xi}$ in $N_{\bb{Q}}/\LL(\Xi)$.  Furthermore, for each inclusion $\Xi\subseteq \Xi'$, we have a toric closed embedding $X(\Sigma_{\Xi'}) \hookrightarrow X(\Sigma_{\Xi})$, and compositions are compatible.  Proposition 3.5 of \cite{NS} says that if $\s{P}$ is integral (i.e., if $\s{P}^{[0]} \subset N$), then $\s{X}_0 \cong \lim_{\stackrel{\longrightarrow}{\Xi\in \s{P}}} X(\Sigma_{\Xi})$.

Now suppose we have the data for a rigid $\f{T}^{\ns}_{g,\Delta}(\A,\Psi)$ as in \S \ref{TropCurves}.  Recall this set is finite by Lemma \ref{General}.  Let $\Sigma$ be any fan whose rays include the rays generated by elements of $\Delta$.  
\begin{lem}
\label{lem-good-subd}
There exists a {\bf good} polyhedral decomposition $\s{P}=\s{P}_{g,\Delta}(\A,\Psi)$---i.e., one which satisfies the following:
\begin{enumerate}
\item The asymptotic fan $\Sigma_{\s{P}}$ is a refinement of $\Sigma$.
\item All the vertices and edges of all tropical curves in $\f{T}^{\ns}_{g,\Delta}(\A,\Psi)$ are contained in the $0$- and $1$-cells of $\s{P}$, respectively. Furthermore, for each compact edge $E$ of such a tropical curve, $\partial^+E-\partial^-E$ is an integral multiple of $w(E)u_E$.
\item Each $A_i\in \A$ is contained in the $\dim(A_i)$-skeleton of $\s{P}$.
\end{enumerate}
Furthermore, by refining $N$, we can assume $\s{P}$ is {\bf integral}, meaning that each vertex $V\in \s{P}^{[0]}$ is contained in $N$.
\end{lem}
\begin{proof}
\cite[Proposition 3.9]{NS} gives us a polyhedral decomposition satisfying the first two properties for any single tropical curve.  Also, any affine subspace $A_i$ of dimension $d$ clearly lives in the $d$-skeleton of some polyhedral decomposition.  $\f{T}^{\ns}_{g,\Delta}(\A,\Psi)$ and $\A$ are finite, and any finite collection of polyhedral decompositions admits a common refinement, so the claim follows.
\end{proof}

We will assume from now on that $\s{P}$ is a good polyhedral decomposition as above.  
By replacing $N$ with $\frac{1}{d}N\subset N_{\bb{Q}}$ for some $d\in \bb{Z}_{> 0}$, we can assume that $\s{P}$ is integral.

\subsection{The moduli space}\label{Moduli}
Let $f^{\dagger}:Y^{\dagger}\rar S^{\dagger}$ be a log smooth morphism of fine saturated log schemes.  \cite{GSlog} and \cite{AC} show that the stack $\s{M}(Y/S)$ of basic\footnote{We will say a bit about the basicness condition at the start of \S \ref{tropicalization}, but we refer to \cite{GSlog} for the precise definition.  The idea is that a morphism $\iota$ of a scheme $T$ to the moduli space $\s{M}(Y^{\dagger}/S^{\dagger},\beta)$ should give a log smooth family of stable log maps over $T$, but there is a priori ambiguity about what log structure $T$ and $\iota$ should have.  Fortunately, there is a ``universal'' choice called the basic or minimal log structure.}  stable log maps to $Y^{\dagger}$ over $S^{\dagger}$ is a Deligne-Mumford log stack locally of finite type over $S$.  Assume $f$ is projective.  Imposing a \emph{combinatorially finite} (see \cite[Def. 3.3]{GSlog}) collection of conditions $\beta$, one obtains a substack $\gls{MYS}$ which is now proper over $S$.  
Furthermore, \cite[Theorem 0.3]{GSlog} says that $\s{M}(Y^{\dagger}/S^{\dagger},\beta)$ has a natural relative perfect obstruction theory and therefore admits a virtual fundamental class
\begin{equation}
\label{eq-vir-rank}
[\s{M}(Y^{\dagger}/S^{\dagger},\beta)]^{\vir}\in A_{\vdim(\s{M}(Y^{\dagger}/S^{\dagger},\beta))+\dim S}(\s{M}(Y^{\dagger}/S^{\dagger},\beta)) 
\end{equation}
which is compatible with base change and equals the usual fundamental class in unobstructed situations, i.e. when of the same degree as the usual fundamental class.  All this makes it possible to define log Gromov-Witten invariants for $Y^{\dagger}$.

Now, equip $\s{X}$ with the divisorial log structure corresponding to $\partial \s{X}$, and equip $\s{X}_t$ with the induced log structure (for $t\neq 0$, this is the divisorial log structure for $\partial X_t$).  
We are interested in the stacks \[\s{M}_{g}(Y^{\dagger},\Delta):=\s{M}(Y^{\dagger}/S^{\dagger},\beta)\] of basic stable log maps to $Y^{\dagger}$ over $S^{\dagger}$ of class $\beta$, where $Y^{\dagger}$, $S^{\dagger}$, and $\beta$ are as follows:  either $Y^{\dagger}=\s{X}^{\dagger}$ and $S^{\dagger}=(\bb{A}^1)^{\dagger}$ (having the divisorial log structure with respect to $\{0\}\subset \bb{A}^1$) or $Y^{\dagger}\ra S^{\dagger}$ is a base-change to a closed point of this: $Y^{\dagger} = \s{X}^{\dagger}_t$ and $S^{\dagger} = \Spec \kk$ if $t\neq 0$ or the standard log point $\gls{kdag}$ if $t=0$.  
The data $\beta$ consists of the following conditions on the basic stable log maps $\varphi^{\dagger}:C^{\dagger}\rar Y^{\dagger}$:
\begin{itemize}
\item $C$ has genus $g$,
\item $\varphi_*[C] = \gls{[Delta]}$ (the class in $A_1(Y,\bb{Z})$ determined by $\Delta$),
\item $C^{\dagger}$ has $e_{\infty}$ marked points $(x_i)_{i\in I^{\circ}},(y_i)_{i\in I^{\partial}}$,
\item For each $i\in I^{\partial}$, $\varphi(y_i)\in D_{u_{(V,E_i)}}$.  Furthermore, if $t_1$ is the generator for the vertical part of the ghost sheaf of $Y^{\dagger}$ at a generic point of $D_{u_{(V,E_i)}}$, and $t_2$ is the generator for the vertical part of the ghost sheaf of $C^{\dagger}$ at $y_i$, then $\varphi^{\flat}:t_1\mapsto w(E_i)t_2$.
\end{itemize}
Geometrically, if the component of $C$ containing $y_i$ is not mapped entirely to the toric boundary, then the last condition means that the $w(E_i)$ equals the intersection multiplicity of $\varphi(C)$ with $D_{u_{(V,E_i)}}$ at $y_i$, see \cite[Ex.7.1]{GSlog}. We call the $\gls{xi}$ {\bf interior marked points} and the $\gls{yi}$ {\bf boundary marked points}.  
The finiteness of types (see Lemma~\ref{Types}) implies that $\beta$ is combinatorially finite in the sense of \cite[Def. 3.3]{GSlog}.  Alternatively, this follows from \cite[Thm 3.8, Examples 3.6 (2) and (3)]{GSlog}.  

Recall $d^{\trop}_{g,\Delta}:=e_{\infty}+(n-3)(1-g)$ from Definition \ref{GeneralDfn}.  
\begin{lem}\label{vdim} If $Y^{\dagger}=\s{X}^{\dagger}$ or $Y^{\dagger}=\s{X}_t^{\dagger}$, we have
$d^{\trop}_{g,\Delta} = \vdim(\s{M}_{g}(Y^{\dagger},\Delta))$.
\end{lem}
\begin{proof}
Since the boundary in a toric variety is anticanonical, and similarly in $\s{X}_0$, 
we have $\omega_{Y^\dagger/S^\dagger}\cong\s{O}_Y$. Hence, $\deg(f^*\omega_{Y^\dagger/S^\dagger}) = 0$.
By log-smoothness, $\Omega_{Y^\dagger/S^\dagger}$ is locally free (in fact globally free).

Using stability and Hirzebruch-Riemann-Roch on curves, it is standard to reduce to 
$$\vdim(\s{M}_{g,e_{\infty}}(Y^{\dagger},\Delta))
=-\deg(f^*\omega_{Y^\dagger/S^\dagger})+\rank(\Omega_{Y^\dagger/S^\dagger})(1-g)+\dim(\?{\s{M}}_{g,e_{\infty}}).$$
We have $\rank(\Omega_{Y^\dagger/S^\dagger})=n$, and it is standard that $\dim(\?{\s{M}}_{g,e_{\infty}})=e_\infty+3g-3$.  The result follows.
\end{proof}

Let 
\[\s{M}^{tt}_{g}(\s{X}_t^{\dagger},\Delta) \subset \s{M}_{g}(\s{X}_t^{\dagger},\Delta)
\]
denote the open subspace of {\bf torically transverse} curves---i.e., curves which do not intersect toric strata of codimension $\geq 2$ in $\s{X}_t$.  For $t=0$, we mean that the intersection of the curve with any irreducible component of $\s{X}_0$ is torically transverse in that component.

\subsubsection{Incidence conditions} \label{Inc}
We continue to assume our incidence conditions have affine tropicalizations, and we defer the more general cases to \S \ref{GenInc}.  Let $A$ be an affine subspace from the data $\A$.  For $\s{P}=\s{P}_{g,\Delta}(\A,\Psi)$ as in Lemma \ref{lem-good-subd}, $A$ together with a choice of point $Q_{A}$ in the big torus orbit of $\s{X}$ determines a subvariety $\gls{ZA}\subset \s{X}$ (suppressing the dependence on $Q_A$ in the notation) defined as the orbit closure
\begin{align*}
Z_{A}:=\?{\bb{G}(\LC(A) \cap N') \cdot Q_A} \subset \s{X}.
\end{align*}
Equivalently, this is the subvariety consisting of the points $x$ such that $z^v(x)=z^v(Q_A)$ for all $v\in M':=\Hom(N',\bb{Z})$ with the property $v|_{\LC(A)}=0$.

Now consider a primitive vector $u\in N\cap \LC(A)$.  We define $\gls{ZAu}\subset D_{u} \subset \s{X}$ by:
\[Z_{A,u}:= Z_{A} \cap D_u.
\]
Note that the components of $D_{u,0}:=D_u\cap \s{X}_0$ correspond to 1-cells of $\s{P}$ which are unbounded in the $u$-direction. 
Similarly as for the $Z_{A}$, $Z_{A,u}$ meets the dense torus of a toric stratum given by a cone $\sigma\in \Sigma'_{\s{P}}$ if and only if $u\in\sigma\subset\LC(A)$.

Given affine incidence conditions $\A=(A_i)_{i\in I}$ for a degree $\Delta$, we specify the $Q_{A_i}$'s generically. 
Note that for each $i\in I^{\circ}$, $\Delta$ determines a primitive vector $u_i$ (with direction $\Delta(i)$) generating a ray of $\Sigma_\s{P}$ that is contained in $\LC(A_i)$.
We will need the regularity statements in the following for the purpose of intersection theory.
\begin{lem} 
\label{lem-reg-emb}
A subvariety $Z_{A_i}$ meets the dense torus of a toric stratum given by a cone $\sigma\in \Sigma'_{\s{P}}$ if and only if $\sigma\subset\LC(A_i)$.  Furthermore, the subvarieties $Z_{A_i}$ are regularly embedded in $\s{X}$ for each $i\in I^{\circ}$, and each $Z_{A_i,u_i}$ is regularly embedded in $D_{u_i}$ for each $i\in I^{\partial}$.  Hence, the analogous statement also holds also in $\s{X}_t$ for each $t$.
\end{lem}
\begin{proof}
By Property (3) of a good polyhedral decomposition (Lemma~\ref{lem-good-subd}), $A_i$ is a union of cells $\Xi\in\s{P}$, and therefore
$Z_{A_i}$ is covered by the toric open sets $U_\Xi:=\Spec\kk[\LC(\Xi)^\vee\cap M']$ for $\Xi\subset A_i$ a maximal cell of $A_i$.  The first statement follows.

Furthermore, since $\LC(\Xi)=\LC(A_i)$, $U_{\Xi}\cap Z_{A_i}$ is the closure of the $Q_{A_i}$-orbit of the stabilizer $\GG(\LC(A_i)\cap N')$ of the minimal toric stratum $O(\Xi):=\Spec(\kk[\LC(\Xi)^\perp\cap M'])$ in $U_\Xi$. 
By this identification, it follows that $\overline{\GG(\LC(A_i)\cap N')\cdot Q_{A_i}}\cap O(\Xi)$ is just a point given by $z^v(x)=z^v(Q_{A_i})$ for $v\in \LC(A_i)^\perp$.  Furthermore, since $V(\Xi)$ is regular (being just an algebraic torus), this point is regularly embedded by a sequence  $z^{v_k}-z^{v_k}(Q_{A_i})$, and this sequence canonically lifts to a regular sequence for $Z_{A_i}$ (we can use the same equations now interpreted as elements of $\kk[\LC(\Xi)^\vee\cap M']$ since this naturally contains $\kk[\LC(\Xi)^\perp\cap M']$ as a subring).
The situation for $Z_{A_i,u_i}$ is the same as for $Z_{A_i}$ except with further intersection with $D_{u_i}$, so the statement follows.
\end{proof}

\subsubsection{Psi-class conditions}
Recall that the degree $\Delta$ is really a map $\Delta:I\rar N$.  Let $\Delta'$ denote the degree $\Delta':I\cup \{0\}\rar N$ defined by $\Delta'(i)=\Delta(i)$ for $i\in I$, and $\Delta'(0)=0$ (i.e., we've added one interior marked point $x_0$).  Consider the map
\[\pi:\s{M}_{g}(Y^{\dagger},\Delta')\rar \s{M}_{g}(Y^{\dagger},\Delta)
\] 
forgetting $x_{0}$---i.e., this is the universal curve over our moduli space.  Let $\sigma_{x_i}$ denote the section corresponding to the marked point $x_i$, $i\in I^{\circ}$.  Let $\omega_{\pi}$ denote the relative cotangent bundle of $\pi$.  Define 
\[\gls{psii}:=c_1(\sigma_{x_i}^* \omega_{\pi}).
\]

 Let
\[\forget:\s{M}_{g}(Y^{\dagger},\Delta) \rar \?{\s{M}}_{g,e_{\infty}}
\]
denote the forgetful map which remembers only the stabilization of the underlying marked curve.  Let $\gls{psiibar}$ denote the $\psi$-class associated with $x_i$ on $\?{\s{M}}_{g,e_{\infty}}$, and define $\gls{psiihat}:=\forget^* \?{\psi}_{i}$.

\begin{prop}\label{psihat}
  $\psi_{i}=\wh{\psi}_{i}$.
\end{prop}
\begin{proof}
From the definitions, the difference $\psi_{i}-\wh{\psi}_{i}$ can be supported on the locus in $\s{M}_{g}(Y^{\dagger},\Delta)$ corresponding to curves $\varphi:C\rar Y^{\dagger}$ for which the component $C_i$ containing $x_i$ is destabilized when $\varphi$ is forgotten.  We claim that this locus is empty.  Indeed, if $C_i$ is destabilized by this forgetful map, then stability of $\varphi$ implies that the map cannot contract $C_i$.  So then $\varphi(C_i)$ must intersect at least two boundary components of $Y$ which do not contain $C_i$.  By Remark 1.9 of \cite{GSlog}, these intersection points $q_1,q_2$ must be nodes or boundary marked points of $C$.  If $q_j$ is a node, let $C^j$ be the component of $C$ intersecting $C_i$ at $q_j$.  If $C^j=C_i$, then $C_i$ is higher genus.  Otherwise, $C^j$ must similarly contain another node or some marked points.  If a node, we can say the same about the component attached to $C^j$ at this node.  Because $C$ can have only finitely many components, this terminates (possibly in a loop, which is stable).  Hence, after forgetting $\varphi$ and stabilizing all components except $C_i$, we see that $q_1$ and $q_2$ are still nodes or marked points or that $C_i$ is higher genus.  Since $C_i$ also contains the marked point $x_i$, we see that $C_i$ is in fact stable.
\end{proof}

\begin{rmk}
\cite{AGr}'s genus zero correspondence theorem is stated in terms of the $\wh{\psi}_{i}$-classes (the associated invariants are sometimes called gravitational {\it ancestors} instead of {\it descendants}).  Proposition \ref{psihat} above shows that for log GW invariants of toric varieties, this is equivalent to the usual version of $\psi$-classes. 
\end{rmk}

We record for later:

\begin{lem} \label{lem-psi-basepointfree}
For any $i$, $\?{\psi}_{i}$ on $\?{\s{M}}_{0,b}$ can be represented by a base-point free divisor, assuming $b>3$.
\end{lem}
\begin{proof}
This follows from the relations $\?{\psi}_i \sim D_{i|jk}$, where $D_{i|jk}$ denotes the boundary divisor corresponding to curves for which the forgetful map remembering only the marked points corresponding to $i,j,k$ (all distinct) destabilizes the component containing $x_i$ (cf.  \cite[Prop 5.1.8]{Kock}).  Indeed, the common intersection of all possible choices of $D_{i|jk}$ is empty. 
\end{proof}

\subsection{The algebraic curve counts}
\begin{dfn} \label{def-log-GW}
Setting
$
\ev_{x_i}: [\varphi^{\dagger}:(C^{\dagger},(x_i)_{i\in I^{\circ}},(y_i)_{i\in I^{\partial}})\rar Y^{\dagger}] \mapsto \varphi(x_i)
$
and similarly for $ev_{y_i}$, we obtain evaluation maps 
$$\ev_{x_i}:\s{M}_{g}(Y^{\dagger},\Delta)\ra \s{X},$$
$$\ev_{y_i}:\s{M}_{g}(Y^{\dagger},\Delta)\ra D_{u_i}.$$
We therefore obtain by Lemma~\ref{lem-reg-emb} via generalized Gysin maps (for $Z_{A_i}$ and $Z_{A_i,u_i}$) and capping with Chern classes (for the $\psi$-classes) a cycle
$$\gls{gamma}=\left(\prod_{i\in I^{\circ}} \psi_{i} \right) \cap \left(\prod_{i\in I^{\circ}} [Z_{A_i}] \right) \cap  \left(\prod_{i\in I^{\partial}} [Z_{A_i,u_i}]\right) \cap [\s{M}(Y^{\dagger}/S^{\dagger},\beta)]^{\vir}$$
that will be a one-cycle on $\s{X}$ and zero-cycle on $\s{X}_t$ by \eqref{eq-vir-rank} and Lemma~\ref{vdim} as we assume \eqref{GeneralDimension}.  
We are interested in its degree, which is the log Gromov-Witten invariant $$\gls{GWlog}:=\deg(\gamma).$$
Let ${}^{\ns}[\s{M}(X_0^{\dagger}/ \Spec \kk^{\dagger},\beta)]^{\vir}$ be the non-superabundant summand of $[\s{M}(X_0^{\dagger}/ \Spec \kk^{\dagger},\beta)]^{\vir}$ under the decomposition theorem \cite[Thm. 1.1.2]{ACGS}, i.e. the sum over those classes that are indexed by non-superabundant curves.
Then define $\gamma^{\ns}$ analogously to $\gamma$ upon replacing $[\s{M}(X_0^{\dagger}/ \Spec \kk^{\dagger},\beta)]^{\vir}$ by ${}^{\ns}[\s{M}(X_0^{\dagger}/ \Spec \kk^{\dagger},\beta)]^{\vir}$. We define
$$\gls{nsGWlog}:=\deg(\gamma^{\ns}).$$
\end{dfn}

As a special case of Theorem \ref{Invariance}, we have:
\begin{prop}\label{GWInv}
$\GW^{\log}_{g,\s{X}^{\dagger}_t,\Delta}(\A,\Psi)$ is independent of $t$.
\end{prop}
In the next section, we will use this proposition to compute the invariants (assuming $\f{T}_{g,\Delta}(\A,\Psi)$ contains no superabundant curves) for general $t$ by computing the invariant at $t=0$.  
More generally, if superabundant curves exist, we prove that the non-superabundant constituent of the log Gromov-Witten invariant for $t=0$ coincides with the non-superabundant tropical count.

\section{Proof of the correspondence theorem}\label{Sing}
%Correspondence in the singular fiber

\subsection{Log curves $\rar$ tropical curves}\label{tropicalization}

We briefly review \cite{GSlog}'s approach to tropicalizing a log curve in $\s{X}_0$ (cf. their Discussion 1.13).  For now we assume that $\s{X}^{\dagger} \rar (\bb{A}^1)^{\dagger}$ is defined as in \S \ref{sec-tordeg} with respect to an integral polyhedral decomposition $\s{P}$ (we will add goodness assumptions later).

Consider a stable log map $\varphi^{\dagger}:(C^{\dagger},{\bf x},{\bf y})\rar \s{X}_0^{\dagger}$ over the corresponding ``basic'' log point 
$\gls{kbas}\ra \Spec \kk^{\dagger}$.  The term {\bf basic} means in particular that the log structure of $\Spec \kk_{\varphi}^{\dagger}$ is determined by a chart $Q\rar \kk$, where $\gls{Q}$ is the monoid defined in Equation (1.14) in \cite{GSlog}.  This $Q$ is roughly defined as follows: let
\begin{align} \label{def-Q}
    \wt{Q}:=\prod_{\eta} P_{\eta} \times \prod_q \bb{N}.
\end{align}
Here, the first product is over the set of generic points $\eta$ of the irreducible components of $C$, and the second product is over the set of nodes $q$ of $C$.  If the minimal toric stratum of $\s{X}$ containing $\varphi(\eta)$ is the stratum corresponding to the cone $\sigma_{\eta}\in \s{P}$, then $\gls{Peta}=\sigma_{\eta}^{\vee}\cap M'$, so $P_{\eta}^{\vee}=\sigma_{\eta}\cap N'$. 
Now $Q:=[\iota(\wt{Q})/R]^{\sat}$, where $\iota$ is the inclusion of $\wt{Q}$ into its groupification, and $R$ is a certain saturated subgroup.  We refer to \cite[Construction 1.16]{GSlog} for the precise definition of $R$, but we won't need this here. 

We will use the description of the dual of $Q$ from \cite[Rem 1.18]{GSlog}:
\begin{equation} \label{def-Qdual}
Q^\vee = \left\{((V_\eta)_\eta,(e_q)_q)\in \bigoplus_\eta P_\eta^\vee\oplus\bigoplus_q \bb{N} \,\mid\, \forall q: V_{\eta_1}-V_{\eta_1} = e_qu_q\right\}.
\end{equation}
In the condition $V_{\eta_1}-V_{\eta_1} = e_qu_q$, by $\eta_1,\eta_2$ we refer to the generic points of components adjacent to a node $q$, and the $u_q\in N'$ are determined by the map on ghost sheaves obtained from $\varphi^{\dagger}:C^\dagger\ra\s{X}_0^\dagger$ at $q$.

Now, a choice of element $\tau\in Q^{\vee}$ that doesn't lie in a proper face of $Q^\vee$ corresponds to a local homomorphism $Q\ra \bb{N}$ and in turn to a choice of pullback of $\varphi^{\dagger}$ to the standard log point.
Such a $\tau$ determines a tropical curve $\Gamma$ in $N':=N\times \bb{Z}$ as follows.  For each $\eta$, $\tau$ determines an element $V_{\eta}\in P_{\eta}^{\vee} = \sigma_{\eta}\cap N'$.  The vertices of $\Gamma$ are the $V_{\eta}$'s.  If $C_{\eta_1}$ and $C_{\eta_2}$ are (not necessarily distinct) irreducible components intersecting at a node $q$, then $V_{\eta_1}$ and $V_{\eta_2}$ are connected by an edge $E_q$ given as $e_qu_q$, and the integral length $|u_q|$ of $u_q$ is the weight $w(E_q)$.
If $C_{\eta}$ has geometric genus $g_{\eta}$ and arithmetic genus $g_{\eta}+b_\eta$ (i.e., it has $b_{\eta}$ nodes), then $V_{\eta}$ will be contained in $b_\eta$ self-adjacent edges, and we define the genus of this vertex as $g_{V_\eta}:=g_{\eta}$. 

If $C_{\eta}$ contains a boundary marked point $y_i$ mapping to $D_u$ with multiplicity $w$, then there is a corresponding unbounded edge $E_i$ emanating from $V_{\eta}$ in the direction $u$ with weight $w$ and marking $\epsilon(i)=E_i$.  Similarly, if $x_i\in C_{\eta}$, then there is a contracted non-compact edge $E_i$ with $\partial E_i=V_{\eta}$.  
The definition of $R$ and the balancing conditions on stable log maps from \cite[\S 1.4]{GSlog} then imply the usual tropical balancing condition.
Finally, whenever a vertex $V_\eta$ is bivalent and of genus zero, we remove it and identify its adjacent edges.  Note that the genus $g$ and degree $\Delta$ of $\Gamma$ agrees with that of $C$.

The tropical curves obtained above are integral tropical curves in $N':=N\times \bb{Z}$ such that all vertices project to the same value $b\in\bb{Z}_{>0}$ in the last component, see \cite[Discussion 1.13]{GSlog}.  Dividing by $b$ brings the curve into $N_{\bb{Q}}=(N_{\bb{Q}},1)\subset N'_{\bb{Q}}$.  This makes $\Gamma$ into a tropical curve in the sense of Def.~\ref{TropCurveDfn}.  We call the original tropical curve {\bf primitive} if the corresponding element $\tau$ in the interior of $Q^\vee$ is primitive. We now see that the set of primitive tropical curves is naturally identified (via division by $b$) with the interior of 
$\QQ_{>0}Q^\vee\cap (N_\bb{Q},1)$.
The cells $\Xi\subset \s{P}$ correspond bijectively with the strata $\s{X}_{\Xi}\subset \s{X}_0$, and for a component $C_{\eta}$ of $C$ with generic element $\eta$ mapping to a stratum $\s{X}_\Xi$, the corresponding possibilities for $\frac1b V_{\eta}\in \QQ_{>0}P_{\eta}^{\vee} \cap (N_{\bb{Q}},1)$ are the rational points in the cell $\Xi$.

Summarizing the above description of tropicalization, we have:
\begin{obs}\label{BasicTrop}
The space $\gls{Tphi}$ of primitive tropical curves associated to $\varphi$ is naturally identified with the affine polyhedral subspace of $\bb{Q}_{\ge 0}Q^{\vee}$ whose component in $\bb{Q}_{\ge 0}P_{\eta}^{\vee}$ is contained in $(N_{\bb{Q}},1)$.
\end{obs}

We obtain the following lemma about when $\varphi$ determines a unique primitive tropical curve.

\begin{lem}\label{BasicTT}
 Assume we are given a basic stable log map
\begin{equation}\label{Basic-logmap}
\begin{tikzcd}
(C^{\dagger},{\bf x},{\bf y}) \arrow{r}{\varphi^{\dagger}} \arrow[swap]{d}{\xi^{\dagger}} &  \s{X}_0^{\dagger} \arrow{d}{\pi^{\dagger}} \\
\Spec \kk^{\dagger}_\varphi  \arrow{r}{\kappa^{\dagger}} & \Spec \kk^{\dagger},
\end{tikzcd}
\end{equation}
where $\Spec \kk^{\dagger}_\varphi$ denotes the basic log point $Q\rar\kk$.  Furthermore, assume that no generic point of a component of $C$ maps to a toric stratum of codimension $\ge 1$ in $\s{X}_0$, and assume that nodes of $C$ only occur between curve components whose generic points map to different components of $\s{X}_0$.
Then the corresponding basic monoid $Q$ is isomorphic to $\bb{N}$. 
Then the map $\kappa^\dagger$ is an isomorphism if and only if for each node $q\in C$, the corresponding edge $h(E_q)$ of $h(\Gamma)$ is of integral lattice length divisible by $w(E_q)$.
\end{lem}
\begin{proof}
Let $\sigma_\eta$ be the cone corresponding to the toric stratum of $\s{X}$ that the generic point $\eta$ of a component of $C$ maps into. 
By the assumption, $\sigma_\eta$ is the cone over a vertex of $\s{P}$, and hence $P_\eta^\vee=\sigma_\eta\cap N'$ is isomorphic to $\bb{N}$, as is $P_\eta$. 
By Observation \ref{BasicTrop}, all primitive tropical curves associated to $\varphi$ have the same vertices, given by the vertices in $\s{P}$ that generate the corresponding $\sigma_\eta$. 
Also, since weights are fixed, the $e_q$ are determined for edges $E_q$ that connect vertices that don't map to the same point in $N$. 
By the assumption on the nodes, this is the case for every edge (i.e., there are no contracted edges). We find there is a unique primitive tropical curve associated to $\varphi$, and then from Observation \ref{BasicTrop} and saturatedness of $Q$ we deduce that $Q=\bb{N}$ 
is generated by this primitive curve, proving the first statement. 

For the second second statement we use the integrality of
$\s{P}$, which implies $\s{X}_0$ is reduced and $\pi^\dagger$ is strict at generic points of components. By definition, $C$ is also reduced and $\xi^\dagger$ is strict on generic points of components. Hence, at these points on ghost sheaf stalks, $\xi^\dagger$ and $\pi^\dagger$ are isomorphisms $\bb{N}\ra\bb{N}$. By commutativity, the map given by $\varphi^\dagger$ on ghost sheaf stalks $\bb{N}=P_\eta\ra Q=\bb{N}$ is isomorphic to that of $\kappa^\dagger:\bb{N}\ra\bb{N}$ which is multiplication by some $b>0$. 
We can compute this number from $P_\eta\ra Q$ as the second component of $V_\eta \in N\oplus\ZZ$. In view of \eqref{def-Qdual}, since the $V_\eta$ are divisible by $b$ using integrality of $\s{P}$, we find $b=1$ if and only if for all edges, $V_{\eta_1}-V_{\eta_2}$ is divisible by $w(E)$ whenever $E$ is an edge connecting $V_{\eta_1}$ and $V_{\eta_2}$.
\end{proof}

\begin{ass}
From now on, unless otherwise stated, we assume we are dealing with a fixed rigid $\f{T}^{\ns}_{g,\Delta}(\A,\Psi)$ with respect to which $\s{P}$ is a good integral polyhedral decomposition.
\end{ass}

We will want to use a forgetful map that remembers slightly more information than the map $\forget$ of Proposition \ref{psihat}.  It is well-known \cite[Prop.\,10.11,\,p.315]{ACG2} that $\?{\s{M}}_{g,n}$ is stratified by dual graphs.  Given a tropical curve $(\Gamma,\epsilon,h)$, recall that $\Aut(\Gamma)$ denotes its automorphisms.  Let $\Aut^{\circ}(\Gamma)$ denote automorphisms of the underlying marked graph $(\Gamma,\epsilon)$ which do not necessarily commute with the weights or the map $h$.  Let $\gls{wtMG}:=\prod_{V\in \Gamma^{[0]}} \?{\s{M}}_{g_V,\val(V)}$.  Then by loc.cit. there is a stratum $\?{\s{M}}(\Gamma)$ of $\?{\s{M}}_{g,e_{\infty}}$ corresponding to $\Gamma$ whose normalization is isomorphic to $[\wt{\s{M}}(\Gamma)/\Aut^{\circ}(\Gamma)]$.  We are interested in the in-between space $\gls{MGam}:=[\wt{\s{M}}(\Gamma)/\Aut(\Gamma)]$.  Note that $\Gamma$ determines a stratum $\s{M}_{g}(\s{X}^{\dagger}_0,\Delta,\Gamma)$ of $\s{M}_{g}(\s{X}_0^{\dagger},\Delta)$ and that we have a well-defined forgetful map $\forget_{\Gamma}$ from this stratum to $\s{M}(\Gamma)$.

\begin{lem}\label{PsiGeomNew}
Consider a rigid $\Gamma\in \f{T}^{\ns}_{g,\Delta}(\A,\Psi)$.  Let $[\pt]$ denote the class of a point in $\s{M}(\Gamma)$. Then
\begin{align*}
   \left.\left(\prod_{i\in I^{\circ}} \psi_{i}^{s_i}\right)\right|_{\s{M}_{g}(\s{X}^{\dagger}_0,\Delta,\Gamma)} = \frac{1}{|\Aut(\Gamma)|} \left(\prod_{V\in \Gamma^{[0]}} \langle V \rangle\right)\forget_{\Gamma}^* [\pt].
\end{align*}
\end{lem}
\begin{proof}
First note that $\psi_{i}=\forget_{\Gamma}^* \psi_{i}$ exactly as in the proof of Proposition \ref{psihat}, so we can interpret the $\psi$-classes as being on $\s{M}(\Gamma)$ rather than on $\s{M}_{g}(\s{X}^{\dagger}_0,\Delta,\Gamma)$.  Now, on $\?{\s{M}}_{0,\val(V)}$ it is known (cf. \cite[Lemma 1.5.1]{Kock}) that $\prod_{i\in I_V^{\circ}} \psi_{i}^{s_i} = \langle V \rangle [\pt]$, so on $\wt{\s{M}}(\Gamma)$ we have $\prod_i \psi_{i}^{s_i} = \left(\prod_{V\in \Gamma^{[0]}} \langle V \rangle\right) [\pt]$.  Since $\psi$-classes on $\wt{\s{M}}(\Gamma)$ are the pullbacks of those on $\s{M}(\Gamma)$, the corresponding intersection on $\s{M}(\Gamma)$ is obtained by dividing by $|\Aut(\Gamma)|$.
\end{proof}

\begin{prop}\label{Log2Trop}
For $\varphi \in \s{M}_{g}(\s{X}_0^{\dagger},\Delta)$ satisfying generic incidence and $\psi$-class conditions corresponding to $\A$ and $\Psi$, we have $\f{T}_{\varphi}\subseteq \f{T}_{g,\Delta}(\A,\Psi)$.  
\end{prop}
\begin{proof}
Suppose that $C_{\eta}$ contains an interior marked point $x_i$ satisfying an incidence condition $Z_{A_i}$.  Then the minimal stratum $\s{X}_{\Xi}$ of $\s{X}_0$ (maximal $\dim\Xi$) that contains $\varphi(x_i)$ must meet $Z_{A_i}$ with its dense torus. By Lemma \ref{lem-reg-emb}, we get $A_i\supset \Xi$. 
Furthermore, $C_\eta$ maps to into a stratum $\s{X}_{\Xi'}$ that contains $\s{X}_{\Xi}$, so $\Xi'\subset\Xi$ and hence $A_i\supset \Xi'$, so $V_\eta\in A_i$.
Thus, if $\varphi$ satisfies the interior incidence conditions $A_i$, then so will all the tropical curves corresponding to points in $\f{T}_{\varphi}$.  Similarly with the boundary incidence conditions.

We now deal with the $\Psi$-conditions. As in the proof of Lemma \ref{PsiGeomNew} above, $\psi$-class conditions on the stratum $\s{M}_{g}(\s{X}^{\dagger}_0,\Delta,\Gamma)$ can be pulled back from $\s{M}(\Gamma)$, and here it is clear from dimension considerations that if $g_V=0$ and $\sum_{i\in I_V^{\circ}} s_i > \dim(\?{\s{M}}_{0,\val(V)}) = \ov(V)$, then $\prod_{i\in I_V^{\circ}}\psi_{i}^{s_i} =0$, as desired.
\end{proof}

\begin{prop}\label{TT}
Let $\s{P}=\s{P}_{g,\Delta}(\A,\Psi)$ be a corresponding good polyhedral decomposition as in Lemma \ref{lem-good-subd}, and assume $N$ is refined enough so that each vertex of $\s{P}$ is contained in $N$.  Then any $\varphi$ mapping to $\s{X}_0^{\dagger}$ satisfying corresponding generic incidence and $\psi$-class conditions and having non-superabundant tropicalization is torically transverse.
\end{prop}
\begin{proof}
By Proposition \ref{Log2Trop}, % and \ref{No-gV}, 
every possible tropicalization of such a $\varphi$ lives in $\f{T}^{\ns}_{g,\Delta}(\A,\Psi)$.  Since $\f{T}^{\ns}_{g,\Delta}(\A,\Psi)$ is finite, $\f{T}_{\varphi}$ must consist of a single tropical curve $\Gamma$ (which we view as living in $(N_\bb{Q},1)\subset N'$), and $Q$ must be isomorphic to $\bb{N}$ by Observation \ref{BasicTrop}.   
If the generic point $\eta$ of $C_V$ for some $V\in\Gamma^{[0]}$ maps to a higher-codimension stratum of $\s{X}_0$ corresponding to the cone $\sigma$ over $\Xi\in \s{P}$, then $\dim \Xi \ge 1$.
Then because the map $P_\eta=\sigma^\vee\cap M' \ra Q=\bb{N}$ on stalks of ghost sheaves is local and is given (up to scaling in $N'$) by $V_\eta\in P_\eta^\vee$, we see $V=V_\eta$ is in the relative interior of $\Xi$, hence not a vertex in $\s{P}$, contradicting the goodness of $\s{P}$.

Suppose a node $q$ between components 
$C_{V_1},C_{V_2}$ with generic points $\eta_1,\eta_2$ mapped into a stratum of $\s{X}_0$ of codimension larger than one.  Then $P_q:=(\varphi^{-1}\overline\shM_{\s{X}_0})_q=\sigma^\vee\cap N'$
for $\sigma$ a cone of dimension at least three containing $P_{\eta_1}^{\vee}$ and $P_{\eta_2}^{\vee}$.  
Since the edge connecting $V_1,V_2$ is part of $\s{P}$ by goodness, so is the two-dimensional cone $\sigma_q$ generated by this edge, and we have a that $\sigma_q$ is a proper face of $\sigma$ (spanned by the rays $P_{\eta_1}^{\vee}$ and $P_{\eta_2}^{\vee}$).

The map $\varphi$ on ghost sheaf stalks together with generization maps from nodes into adjacent components induce a commutative diagram (see also \cite[Discussion 1.8, p.459]{GS})
\begin{equation} \label{butterfly}
\begin{split}
\xymatrix@C=20pt
{&P_{\eta_1}\ar^{\varphi_{\eta_1}}[rr]&&Q\\
P_q\ar^{\varphi_{ q}}[rr]\ar^{\chi_1}[ru]\ar_{\chi_2}[rd]
&&Q\oplus_\NN\NN^2 \ar[ru]\ar[rd]\ar^{\iota}@{^{(}->}[r]
&\ar[u]_{\op{pr}_1} \ar[d]^{\op{pr}_2} Q\times Q\\
&P_{\eta_2}\ar_{\varphi_{\eta_2}}[rr]&&Q.
}
\end{split}
\end{equation}
The map $\varphi_{ q}$ is local and the top and bottom horizontal maps are isomorphisms, so we get a local map
$\iota\circ \varphi_{ q} : P_q\ra P_{\eta_1}\oplus P_{\eta_2}\cong\bb{N}^2$. 
But this map factors through the localization $P_q\ra\sigma_q^\vee\cap N'$, and this contradicts locality since the elements of $\sigma_q^{\perp}\cap \sigma^{\vee}$ map to $0$.
\end{proof}

\subsection{Tropical curves $\rar$ log curves}\label{Trop2Log}

As in \cite{NS}, a {\bf marked pre-log curve} is a stable map $\varphi:(C,{\bf x},{\bf y})\rar \s{X}_0$ such that:
\begin{enumerate}
\item Each irreducible component of $C$ is torically transverse, and
\item For $P\in C$, we have $\varphi(P) \in D\subset \Sing(\s{X}_0)$ for some irreducible component $D$ of $\Sing(\s{X}_0)$ if and only if $P$ is a node of $C$.  Furthermore, the two branches of $C$ at $P$ map to different irreducible components of $\s{X}_0$ and have the same intersection multiplicity with $D$ at $P$.
\end{enumerate}
By Proposition \ref{TT}, any $\varphi^{\dagger}$ satisfying our generic incidence and $\psi$-class conditions with non-superabundant tropicalization must be torically transverse.  Condition (2) is necessary for torically transverse $\varphi$ to admit a log structure over $\s{X}_0$.  We therefore focus on counting marked pre-log curves.

Recall the number $\f{D}_{\Gamma}:=\inde(\Phi) \prod_{V\in \Gamma^{[0]}} \langle V \rangle$ as defined in \eqref{DGamma}.  We want to show that $\f{D}_{\Gamma}$ equals the number of marked pre-log curves in $\s{X}_0$ which satisfy the desired incidence and psi-class conditions and have tropicalization $\Gamma$. We will need the following lemma about torically transverse stable log maps to a toric variety $Y^{\dagger}$ (with divisorial log structure corresponding to its toric boundary).  Here, we will let $N$ be the co-character lattice of $Y$ and let $\GG(N)$ be the big torus orbit of $Y$.  Interior marked points will be denoted by $x_i$, $i\in I^{\circ}$, and boundary marked points by $y_i$ with $\Delta(i)=w_iu_i$.  We assume the fan for $Y$ is refined enough to include the rays generated by these $u_i$'s.  Recall that the maximal torus orbit of the divisor to which $y_i$ maps is identified with $\GG(N/\bb{Z}u_i)$.

\begin{lem}\label{SingleMark} 
For a degree $\Delta$ with $\#I^{\circ} \geq 0$, let $\s{M}$ denote the space of irreducible torically transverse stable maps $\varphi:\bb{P}^1\rar Y$ which can be obtained by forgetting the log structure of a map corresponding to a point of $\s{M}^{tt}_{0}(Y^{\dagger},\Delta)$.  Then $\s{M}\cong \GG(N)\times \s{M}_{0,e_{\infty}}$, with the map to the $\GG(N)$-factor being given by evaluation at $x_{i_0}$ (for any fixed choice of $i_0$), and the map to the $\s{M}_{0,e_{\infty}}$-factor being the forgetful map remembering only the stabilization of the underlying domain curve.  Furthermore, for all $\varphi \in \s{M}$ with the same fixed projection to $\s{M}_{0,e_{\infty}}$, there is a $\tau_j\in\GG(N/u_j)$ such that
the map $\varphi\mapsto\varphi(y_j)$ is given by the image of $\varphi$ under the composition of the projections $\s{M} \rar \GG(N) \rar \GG(N/u_j)$
with multiplication by $\tau_j$. Similarly, for each $i\in I^{\circ}$, there are $\tau_{x_{i}}\in\GG(N)$ such that $\ev_{x_{i}}$ is given as 
$\GG(N)\stackrel{\cdot\tau_{x_{i}}}\longrightarrow\GG(N)$.
\end{lem}
We note that the first statement is \cite[Prop. 3.3.3]{Ran}.  We repeat the proof of this part to set up for the latter statement.

\begin{proof}
We use a slight modification of the quotient construction of toric varieties.  Choose primitive vectors $\{u_j\}_{j\in I'} \in N$ ($I'$ some finite index set) such that $\{u_i|i\in I'\sqcup I^{\partial}\}$ generates $N$.  Let $\wt{I} = I' \sqcup I^{\partial}$.  Define a lattice $\wt{N}\cong \bb{Z}^{\# \wt{I}}$ with a basis denoted by $\{\wt{u_i}|i\in \wt{I}\}$.  Consider the map $\pi:\wt{N}\rar N$, $\wt{u_i}\mapsto u_i$.  Let $\Sigma$ denote the fan in $N$ consisting of the rays generated by the $u_i$'s, $i=1,\ldots,s$, and let $\wt{\Sigma}$ be the fan in $\wt{N}$ consisting of the rays generated by the corresponding $\wt{u_i}$'s.  Then $\pi$ induces a map of fans and hence a map of the corresponding toric varieties $\TV(\wt{\Sigma}) \rar \TV(\Sigma)$.  This can be viewed as a quotient by $\bb{G}(\ker(\pi))$ onto a dense open subset of $Y$ that contains all images of curves in $\s{M}$.

Let $\{v_i|i\in \wt{I}\}$ denote the dual basis to $\{\wt{u_i}|i\in \wt{I}\}$.  The coordinates $z^{v_i}$ on $\TV(\wt{\Sigma})$ can be viewed as homogeneous coordinates on $\TV(\Sigma)$.  In terms of these coordinates, any $\varphi \in \s{M}^{tt}_{0}(Y^{\dagger},\Delta)$ can (similarly to in the proof of \cite[Prop. 3.2]{GPS}) be expressed as 
\begin{align*}
z^{v_i}\mapsto \varphi_i:=\begin{cases}
c_i(X+a_i Y)^{w_i} &\mbox{if } i\in I^{\partial} \\
c_i &\mbox{if } i\in I'
\end{cases}
\end{align*}
where $X$ and $Y$ are the homogeneous coordinates on our domain $\bb{P}^1$ and $c_j=\varphi_j(x_{k})\in \kk^*$ for a fixed choice of $k\in I^{\circ}$.  Here, $x_{k}=(1:0)$, and $y_i=(-a_i:1)$ for $i\in I^{\partial}$.  Since the underlying marked curve $C$ is irreducible by assumption, the $x_i$'s and $y_i$'s are all distinct.  Furthermore, all the $x_i$'s and $y_i$'s, and hence also the $a_i$'s, are uniquely determined by the projection of $\s{M}$ to $\s{M}_{0,e_{\infty}}$, up to a reparametrization of $\bb{P}^1$ which we can fix by, say, fixing $a_1$ and $a_2$. The $c_i$'s are now the only information still needed to specify $\varphi$.  Since two different choices of the $c_i$'s determine the same $\varphi$ exactly if they differ by an element of $\bb{G}(\ker \pi)$, we find that the locus in $\s{M}$ of stable maps with fixed underlying marked curve is $\bb{G}(\wt{N})/\GG(\ker \pi) \cong \bb{G}(\wt{N}/\ker \pi) \cong \GG(N)$.  This proves the first claim.

Next, note that $\varphi_i(y_j) = c_i(-a_j+a_i)^{w_j}$ for $i\in I^{\partial}$, and $\varphi_i(y_j)=c_i$ for $i\in I'$.  Hence, for specified $a_i$'s (determined by specifying the projection of $\varphi$ to $\s{M}_{0,e_{\infty}}$), $\varphi(y_j)$ is linear in the $c_i$'s for $i\neq j$ and uniquely determined modulo the $\bb{G}(\ker(\pi))$-action.  Hence, specifying $\varphi(y_j)\in \GG(N/\bb{Z}u_j)$ is equivalent to specifying the image of $\varphi$ under the projection of $\bb{G}(\wt{N})/\GG(\ker \pi) \cong \GG(N)$ to $\bb{G}(N/\bb{Z}u_j)$, and linearity gives that the two $\bb{G}(N/\bb{Z}u_j)$'s are related by multiplication by some $\tau_j$, as claimed. This is similar for the $\tau_{x_{k}}$.
\end{proof}

We will also need the following analogous statement (essentially \cite[Proposition 5.5]{NS}) for curves with no interior marked points:\newpage
\begin{lem}\label{NoMarks}
Suppose $I^{\circ}=\emptyset$. 
\begin{enumerate}
\item If $e_{\infty}=3$, then $\s{M}^{tt}_{0}(Y^{\dagger},\Delta)$ is a torsor over $\bb{G}(N)$, and, after choosing a trivialization, the map $\s{M}^{tt}_{0}(Y^{\dagger},\Delta)\stackrel{\ev_i}{\rar} \bb{G}(N/\bb{Z}u_i)$,
$\varphi\mapsto\varphi(y_i)$ is given by the projection $\bb{G}(N) \rar \bb{G}(N/\bb{Z}u_i)$ up to composing with the action of some $\tau_i\in \bb{G}(N/\bb{Z}u_i)$.
\item If $e_{\infty}=2$, then the locus of irreducible curves in $\s{M}^{tt}_{0}(Y^{\dagger},\Delta)$ is naturally identified with $\bb{G}(N/\bb{Z}u_i)$ ($i$ can be either $1$ or $2$).  Furthermore, under this identification, $\s{M}^{tt}_{0}(Y^{\dagger},\Delta)\stackrel{\ev_i}{\rar} \bb{G}(N/\bb{Z}u_i)$,
$\varphi\mapsto\varphi(y_i)$ is the identity map.  
\end{enumerate}
\end{lem}
\begin{proof}
We want to deduce (1) from Lemma~\ref{SingleMark} by first considering the case of $e_\infty=4$ with $|\partial I|=3$ and $|I^\circ|=1$, so that Lemma~\ref{SingleMark} gives $\s{M}^{tt}_{0}(Y^{\dagger},\Delta)\cong \bb{G}(N)\times \s{M}_{0,4}$. The subspace inside $\s{M}^{tt}_{0}(Y^{\dagger},\Delta)$ where we fix the four markings (for some fixed choice of ordering) to be, say, $0$, $1$, $2$ and $\infty$ is identified with the space of only three markings and also with $\bb{G}(N)$ under the isomorphism which is what was claimed.

For case (2), we must have $u_1=-u_2$, say each with index $w$.  For $\Sigma$ the fan whose only rays are generated by $u_1,u_2$, torically transverse $\varphi$ must map to $\TV(\Sigma)\subset Y^{\dagger}$. In fact, $\TV(\Sigma)\cong \bb{P}^1\times \bb{G}(N/\bb{Z}u_i)$, and $\varphi$ must be a degree $w$ cover for one of the $\bb{P}^1$-fibers of the projection $\pi$ to $\bb{G}(N/\bb{Z}u_i)$.  Furthermore, this cover of $\bb{P}^1$ must be totally ramified at $0$ and $\infty$ over the two points meeting the toric boundary, and this completely determines the cover up to isomorphism.  So such maps $\varphi:C\rar \TV(\Sigma)$ are uniquely determined by the point $\pi(\varphi(C))\in \bb{G}(N/\bb{Z}u_i)$, and this point is clearly equal to the image under either evaluation map $\ev_i$.
\end{proof}

We are ready to describe the space of pre-log curves in $\s{X}_0^{\dagger}$ associated to a specified rigid tropical curve.  Given $\Gamma$, let $\Phi^{\circ}$ denote the map of \eqref{D} in the case where each condition $A_i$ is trivial (i.e., all of $N$).  That is,
\begin{align*}
\Phi^{\circ}:=\prod_{V\in \Gamma^{[0]}} N &\rar \prod_{E\in \Gamma^{[1]}_c} N/\bb{Z}u_{\partial^-E,E} \\
H &\mapsto ((H_{\partial^+E}-H_{\partial^-E})_{E\in \Gamma^{[1]}_c}). \notag
\end{align*}
We denote $\Phi^{\circ}_{\bb{Q}}:=\Phi^{\circ} \otimes \bb{Q}$ as in Lemma \ref{TropDeform},  $\Phi^{\circ}_{\kk^*}:=\Phi^{\circ}\otimes \kk^*$, and $\Phi_{\kk^*}:=\Phi \otimes \kk^*$.

\begin{lem} \label{lemma-G-tensor}
For a homomorphism $\Phi:L\ra L'$ of free Abelian groups, consider the map $\Phi_{\kk^*}:\GG(L)\ra\GG(L')$ obtained by tensoring $\Phi$ with $\kk^*$. One has an exact sequence
$$ 0\ra \GG(\ker\Phi)\ra \ker(\Phi_{\kk^*}) \ra \op{Tor}^1_\ZZ(\coker\Phi,\kk^*)\ra 0.$$
\end{lem}
\begin{proof}
This follows from applying the spectral sequence of the right exact functor $\kk^*\otimes\cdot$ to the exact sequence 
$0\ra \ker\Phi\ra L\ra L'\ra \coker\Phi \ra 0$.
\end{proof}

Recall the notation $\wt{\s{M}}(\Gamma)$ and $\s{M}(\Gamma)$ of Lemma \ref{PsiGeomNew}. 

\begin{prop}\label{PreCount}
Let $\gls{Mprelog}$ denote the space of pre-log curves corresponding to stable log maps in $\s{M}^{tt}_{g}(\s{X}_0^{\dagger},\Delta,\Gamma)$ with tropicalization equal to the rigid tropical curve $\Gamma$.  Then $\s{M}^{\mathrm{pre-log}}_{g}(\s{X}_0^{\dagger},\Delta,\Gamma)$ is a $\ker(\Phi^{\circ}_{\kk^*})$-torsor over $\s{M}(\Gamma)$.  In particular, $\s{M}^{\mathrm{pre-log}}_{g}(\s{X}_0^{\dagger},\Delta,\Gamma)$ is a smooth Deligne-Mumford stack of dimension $d^{\trop}_{g,\Delta} = \vdim(\s{M}_{g}(\s{X}^{\dagger}_t,\Delta))$. 
Furthermore, 
the intersection of the $[Z_{A_i}]$'s, $[Z_{A_i,u_i}]$'s, and general representatives of the $\psi$-conditions in $\s{M}^{\mathrm{pre-log}}_{g}(\s{X}_0^{\dagger},\Delta,\Gamma)$ associated to $\A$ and $\Psi$ consists of $\f{D}_{\Gamma}$ points, each with multiplicity $\frac{1}{|\Aut(\Gamma)|}$.  
\end{prop} 
\begin{proof}
Recall that we have a projection $\forget_{\Gamma}: \s{M}^{\mathrm{pre-log}}_{g}(\s{X}_0^{\dagger},\Delta,\Gamma) \rar \s{M}(\Gamma)$.  Consider the pullback  $\forget'_{\Gamma}$ of $\forget_{\Gamma}$ to $\wt{\s{M}}(\Gamma)$.  We will describe the torsor structure over this pullback.

Let $\wt{\Gamma}$ be the modification of $\Gamma$ obtained by inserting a vertex at every point of $\Gamma$ mapping to a vertex of our polyhedral decomposition $\s{P}$.  Note that these new vertices are unmarked and bivalent.  
Consider the ``unglued'' space 
\begin{align}\label{NoGlue}
\s{M}^{\circ}(\Gamma) :=\prod_{V\in \Gamma^{[0]}} \left(\GG(N)\times\s{M}_{0,\val(V)}\right)\ \times \prod_{V\in \wt{\Gamma}^{[0]}\setminus \Gamma^{[0]}} \GG(N/\bb{Z}u_V)
\end{align}
where $u_V$ denotes the direction of the edge of $\Gamma$ passing through $V$. By means of 
Lemmas \ref{SingleMark} and \ref{NoMarks}, we find that $\s{M}^{\circ}(\Gamma)$ is isomorphic to the product of the moduli spaces of torically transverse maps $(\varphi^V:C_V\rar \s{X}_0)_{V\in\Gamma^{[0]}}$ to the respective components in $X_0$ as prescribed by $\Gamma$. 

We now want to understand when the components $\varphi^V:C_V\rar \s{X}_0$ glue together to form a pre-log curve.  
For this purpose, we fix the projection of $\s{M}^{\circ}(\Gamma)$ to $\wt{\s{M}}(\Gamma)\cong\prod_{V\in \Gamma^{[0]}} \s{M}_{0,\val(V)}$

Suppose that $V_1,V_2\in \Gamma^{[0]}$ are connected by an edge $E\in \wt{\Gamma}^{[1]}$.  By Lemma \ref{SingleMark}, the component $\varphi^{V_i}$ corresponding to $V_i$ corresponds to a point in $\bb{G}(N_i)$, where $N_i:=N$.  We use this subscript $i$ in $N_i$ to distinguish the two copies of $\bb{G}(N)$.
By Lemmas~\ref{SingleMark} and \ref{NoMarks}(1) the evaluation maps to $\bb{G}(N/\bb{Z}u_{E})$ are given by
$\bb{G}(N_i)\rar\bb{G}(N_i/\bb{Z}u_{E})\stackrel{\cdot\tau_i}\rar\bb{G}(N/\bb{Z}u_{E})$
for $\tau_1,\tau_2\in\bb{G}(N/\bb{Z}u_{E})$. Hence, $\varphi_{V_1}$ and $\varphi_{V_2}$ glue if and only if $\tau_1 \?{\varphi}_{V_1} = \tau_2 \?{\varphi}_{V_2}$, or equivalently, if and only if $(\varphi_{V_1},\varphi_{V_2})$ lives in the fiber over $\tau_1^{-1}\tau_2$ for the map 
\begin{align} \label{twistkernel}
\bb{G}(N_1)\oplus \bb{G}(N_2)    &\rar \bb{G}(N/\bb{Z}u_{E})\\
           (\varphi_1,\varphi_2) &\mapsto \?{\varphi}_1\?{\varphi}_2^{-1}. \nonumber
\end{align}

More generally, suppose $V_1,V_k\in \Gamma^{[0]}$ are connected by an edge $E\in \Gamma^{[1]}$, but there is some sequence $V_2,\ldots,V_{k-1} \in \wt{\Gamma}^{[0]}$ on $E$ between $V_1$ and $V_k$.
Similarly to before, we now have that the chain of curves $C_{V_i}$ glues if the corresponding vector $(\varphi_1,...,\varphi_k)$ in the first row of

\resizebox{0.95\textwidth}{!}{
\xymatrix{
\bb{G}(N_1)\ar[rd]^{\tau_1\cdot}&& \ar[ld]_{\id}\GG(N_2/\bb{Z}u_{(V_2,E)})\ar[rd]^{\id} &&\ldots&& \ar[dl]_{\id}\GG(N_{k-1}/\bb{Z}u_{(V_{k-1},E)})\ar[rd]^{\id} &&\ar[ld]_{\tau_k\cdot}\bb{G}(N_k)\\
&\GG(N/\bb{Z}u_{E})&&\ldots&&\GG(N/\bb{Z}u_{E})&&\GG(N/\bb{Z}u_{E})
}
}

\noindent has the property that its images under the two possible maps to each factor in the bottom row agree. Equivalently, this means that $\tau_1 \?{\varphi}_1 = \varphi_2 = \ldots = \varphi_{k-1} = \tau_k \?{\varphi}_k$. 
This allows us to remove the $\GG(N/\bb{Z}u_V)$-terms from \eqref{NoGlue}, instead just directly imposing the condition that $(\varphi_1,\varphi_k)$ is in the fiber over $\tau_1^{-1}\tau_k$ for the map  $\bb{G}(N_1)\oplus \bb{G}(N_k) \rar \bb{G}(N/\bb{Z}u_{E})$, $(\varphi_1,\varphi_k) \mapsto \?{\varphi}_1\?{\varphi}_k^{-1}$  as with \eqref{twistkernel} before.

With these $\GG(N/\bb{Z}u_V)$-terms removed and for fixed 
projection to $\wt{\s{M}}(\Gamma)$,
the right-hand side of \eqref{NoGlue} reduces to precisely the domain of $\Phi^{\circ}_{\kk^*}$.  Furthermore, we see that the gluing conditions collectively are precisely the condition that $\varphi$ lives in a specific fiber of $\Phi^{\circ}_{\kk^*}$.  By rigidity and Lemma~\ref{phi-surjective}, $\Phi^{\circ}_{\kk^*} = \Phi^{\circ}_{\bb{Q}} \otimes \kk^*$ is surjective, so each fiber of $\Phi^{\circ}_{\kk^*}$ is a coset of $\ker(\Phi^{\circ}_{\kk^*})$.  This reasoning works for all fixed curves in $\wt{\s{M}}(\Gamma)$, so this proves the first claim.  In particular, $\forget'_{\Gamma}$ is onto.

For the second claim, Proposition \ref{TropDeform} says that the kernel of $\Phi^{\circ}_{\kk^*}$ has dimension $n+\?{e}-ng$.  The target of $\forget'_{\Gamma}$ has dimension $\sum_{V\in \Gamma^{[0]}} \dim(\s{M}_{0,\val(V)}) = \sum_V \ov(V) = \ov(\Gamma)$.   Hence, the total dimension for $\s{M}^{\mathrm{pre-log}}_{g}(\s{X}_0^{\dagger},\Delta,\Gamma)$ is $n+\ov(\Gamma)+\?{e}-ng$.  This indeed equals $d^{\trop}_{g,\Delta}$ by \eqref{dtrop}, and this, in turn, equals $\vdim(\s{M}_{g}(\s{X}^{\dagger}_t,\Delta))$ by Lemma \ref{vdim}.

Finally, we impose the incidence and $\psi$-class conditions.  From Lemma \ref{PsiGeomNew}, $|\Aut(\Gamma)|$ times the intersection product of the $\psi$-class conditions determines the image of $\forget_{\Gamma}$ up to choosing one of $\prod_{V\in \Gamma^{[0]}} \langle V\rangle$ possible points (we can use genericity of the $\psi$-conditions, the base-point-freeness from Lemma~\ref{lem-psi-basepointfree}, and Bertini's theorem to say we have $\langle V\rangle$ distinct reduced points).  
Once an image for $\forget_{\Gamma}$ is fixed, we have established that the space of possible pre-log curves without the incidence conditions is $\ker\bb{G}(\Phi^{\circ})$. 
Recall that we have maps 
$\GG(N) \stackrel{\tau_{x_i}}\longrightarrow\GG(N)$ and 
$\GG(N)\ra\GG(N/\bb{Z}u_j)\stackrel{\tau_j}\longrightarrow\GG(N/\bb{Z}u_j)$ giving 
$\ev_{x_i}$ and $\ev_{y_j}$, respectively, when applied to the component $\varphi_V\in\GG(N)$ with $C_V$ containing $x_i$ or $y_j$, respectively.  So meeting the incidence conditions is equivalent to having 
$\tau_i\cdot \varphi_{\partial E_i}\in Z_{A_i}$ and
$\tau_j\cdot [\varphi_{\partial E_j}~(\mathrm{mod\, } \bb{G}(\bb{Z}u_j))]\in Z_{A_j,u_j}$ for each $i$ and $j$.  
 These conditions are equivalent to 
$\varphi_{\partial E_i}\in \tau_{x_i}^{-1}Z_{A_i}$
and
$\varphi_{\partial E_j}~(\mathrm{mod\, }\bb{G}(\bb{Z}u_j))\in \tau_j^{-1}Z_{A_j,u_j}$. 

The set of tuples $(\varphi_V)_{V\in \Gamma^{[0]}}\in \domain(\Phi_{\kk^*}^{\circ}) = \domain(\Phi_{\kk^*})$ satisfying this and the gluing conditions is the intersection of a set of orbits under the translation actions of the subgroups  
$\ker \Phi^{\circ}_{\kk^*}$, $\{(\varphi_V)\in \domain(\Phi_{\kk^*}): \varphi_{\partial E_i}\in \LL(A_i)\}$ for $i\in I^{\circ}$, and $\{(\varphi_V)\in \domain(\Phi_{\kk^*}): \varphi_{\partial E_j}\in \LL(A_j)\}$ for $j\in I^{\partial}$. 
The intersection of all these subgroups is just $\ker(\Phi_{\kk^*})$, and we may now apply
Lemma~\ref{lemma-G-tensor} 
using the injectivity of $\Phi$ from Lemma~\ref{Ddfn} to identify this in turn with
$$\op{Tor}_\ZZ(\coker(\Phi),\kk^*).$$
In particular, the intersection of the subgroups is a union of $|\coker \Phi|$ reduced points using that $\kk$ is algebraically closed of characteristic zero to have $\op{Tor}(\ZZ/p,\kk^*)\cong \ZZ/p$.

Using the homogeneity of coset intersections, we indeed obtain that
$|\Aut(\Gamma)|\cdot\bigcap_i\psi_i^{s_i}[Z_{A_i}]\cap\bigcap_j[Z_{A_i,u_i}]$
consists of $\mathrm{index}(\Phi)\prod_V\langle V\rangle=\f{D}_{\Gamma}$ many reduced points.
\end{proof}

We next describe the number of log structures that we can put on each of our pre-log curves.

\begin{lem}\label{LogStructureCounts}
Let $(\varphi:C\rar \s{X}_0,\bf{x},\bf{y})$ be a marked pre-log curve with associated tropical curve $\Gamma$.  Assume that for every compact edge $E\in \wt{\Gamma}^{[1]}_c$, the integral length $\ell(E)$ of $h(E)$ is a multiple of its weight $w(E)$ (always achievable by rescaling $N$).  Then there are exactly $\prod_{E\in \Gamma^{[1]}_c} w(E)$ isomorphism classes of diagrams
\begin{equation}\label{Basic}
\begin{tikzcd}
C^{\dagger} \arrow{r}{\varphi^{\dagger}} \arrow[swap]{d}{\xi^{\dagger}} &  \s{X}_0^{\dagger} \arrow{d}{\pi^{\dagger}} \\
\Spec \kk^{\dagger}_\varphi  \arrow{r}{\kappa^{\dagger}} & \Spec \kk^{\dagger}
\end{tikzcd}
\end{equation}
where $\kappa^{\dagger}:\Spec \kk^{\dagger}_\varphi \rar \Spec \kk^{\dagger}$ has the basic log structure associated to $\varphi$.
\end{lem}
\begin{proof}
By Lemma \ref{BasicTT}, $\Spec \kk^{\dagger}_{\varphi}$ is actually the standard log point $\Spec \kk^{\dagger}$, and $\kappa^{\dagger}$ is an isomorphism.  Given this, the Lemma becomes a restatement of \cite[Prop. 7.1]{NS} (possibly in higher genus, but this has no effect on the argument in loc.cit.).
\end{proof}

In order to relate our curve-counting to the log Gromov-Witten numbers, we need the following:
\begin{prop}\label{VirtualActual}
$[\s{M}_{g}(\s{X}_0^{\dagger},\Delta)]$ and $[\s{M}_{g}(\s{X}_0^{\dagger},\Delta)]^{\vir}$ agree in an open neighborhood of any torically transverse basic  stable log map $C$ corresponding to a non-superabundant $\Gamma$.
\end{prop}
This is equivalent to saying that if $\Gamma$ is non-superabundant, then the corresponding torically transverse basic stable log map is unobstructed.  As pointed out to us by Dhruv Ranganathan, this (along with the converse) is the content of \cite[Prop. 4.2]{CFPU}.  Given our setup, we offer the following very different proof:
\begin{proof}
We saw in Proposition \ref{PreCount} that a neighborhood of $C$ in the space of pre-log curves is smooth of dimension equal to $\vdim(\s{M}_{g}(\s{X}_0^{\dagger},\Delta))$.  The claim now follows because, as we saw in Lemma \ref{LogStructureCounts}, a neighborhood of the point corresponding to $C$ in the log moduli space is finite over this space of pre-log curves.
\end{proof}

\begin{lem}\label{NoAut}
Let $C$ be a basic stable log map obtained from equipping a pre-log curve cut out by generic rigid incidence and $\psi$-class conditions as in Lemma \ref{PreCount} with a log structure as in Lemma \ref{LogStructureCounts} (suppressing the map, markings, and log structure from the notation). 
Then $C$ admits no non-trivial automorphisms (using from Def.~\ref{TropCurveDfn} that $\Gamma^{[0]}\neq \emptyset$).
\end{lem}

\begin{proof}
Let $\Gamma$ be the rigid tropical curve corresponding to $C$. First note that, being a pre-log curve, $C$ has no contracted components and each component is a rational curve without self-intersections. We look at possible automorphisms for each component.

First assume $C_V$ is a component with a $\psi$-condition, i.e. $C_V$ has at least four special points (marked points and points where it forms a node with adjacent components) and at least one of these is a marking.  A non-trivial automorphism of $C_V$ needs to permute the special points and, since $C_V\cong\bb{P}^1$, specifying where three of these points go determines the entire automorphism already, i.e. the permutation of the points that aren't markings determines the automorphism.  But then since the markings need to be fixed, these need to be among the finite set of fixed points of such an automorphism.  By the genericity of the $\psi$-conditions, this can't happen, so $C_V\ra\s{X}_V$ has no non-trivial automorphisms.

Now suppose $V$ is an unmarked trivalent vertex of $\Gamma$. The corresponding $C_V$ has three special points, now all nodes or boundary marked points. If these map to different points in the target, then $C_V$ has no non-trivial automorphisms because an automorphism of $\bb{P}^1$ that fixes three points is the identity.
If $\varphi_V:C_V\ra\s{X}_V$ is not injective on these points, then $V$ must have the form 
\xymatrix{
\ar@{-}[r]^(-.15){}="a"^(1.15){}="b" \ar@{-} "a";"b"
&\bullet
\ar@{=}[r]^(-.1){}="a"^(1.1){}="b" \ar@{=} "a";"b"
&}
in $\Gamma$ with some suitable weights, and such an unmarked vertex can of course be slid along the edges in either direction, meaning that such a $\Gamma$ would not be rigid.

Similarly, when $V$ is a marked bivalent vertex, there can be no nontrivial automorphisms because $\varphi_V:C_V\ra\s{X}_V$ is injective on the three special points.

Finally, consider an unmarked bivalent vertex $V$ of $\wt{\Gamma}$ as in the proof of Lemma \ref{PreCount}.  Let $w$ be the weight of the edge $E$ of $\Gamma$ containing $V$.  Then the corresponding component $C_V$ of the pre-log curve has $w$ automorphisms $z\mapsto \zeta_w^k z$, $k=0,\ldots,w-1$, which commute with the map $C_V\rar \s{X}_V$, $z\mapsto z^w$. 
These automorphisms permute the $w$ non-isomorphic choices for the log structure at the nodes of $C_V$, cf. proof of \cite[Prop. 7.1]{NS}, unless this automorphism is matched with a corresponding one on the neighboring component of the node (uniquely determined near the node). Thus, for such an automorphism to lift to one on all of $C^\dagger$, since all higher-valent or marked vertex components have no automorphisms, $\wt{\Gamma}$ needs to be a straight line whose only vertices are unmarked bivalent vertices.  But then $\Gamma^{[0]}=\emptyset$, and we excluded such cases in Definition \ref{TropCurveDfn} (cf. Remark \ref{NoVertices}).  
\end{proof}

\begin{rmk} \label{rmk-pre-log-aut}
Note that the statement of Lemma~\ref{NoAut} doesn't hold for pre-log curves.  As we saw in the proof, an unmarked bivalent vertex $V$ of $\Gamma$ contained in an edge $E$ contributes $1/w(E)$ automorphisms.  
\end{rmk}

We are now ready to prove our main theorem:
\begin{thm}\label{MainThm}
Assuming rigidity, we have
$${}^{\ns}\GW_{g,N_{\bb{Q}},\Delta}^{\trop}(\A,\Psi) = {}^{\ns}\GW^{\log}_{g,\s{X}^{\dagger}_0,\Delta}(\A,\Psi).$$
If the superabundant part of this invariant vanishes, then it also coincides with $\GW^{\log}_{g,\s{X}^{\dagger}_t,\Delta}(\A,\Psi)$ for any $t$.
\end{thm}
\begin{proof}
The right-hand side is by Definition~\ref{def-log-GW} the degree of a zero-cycle $\gamma^{\ns}$ built from a generic choice of the incidence and $\psi$-class conditions. The virtual fundamental class involved in the definition of $\gamma^{\ns}$ is by Proposition~\ref{VirtualActual} the usual fundamental class.  Let $\forget_{\log}$ denote the forgetful map to the coarse moduli space of pre-log curves. 
The cycle $\gamma^{\ns}$ equals $\forget_{\log}^* \gamma'$ for some zero-cycle $\gamma'$ on the space of pre-log curves, and Proposition~\ref{Log2Trop} shows that the tropicalization map $f^{\trop}$ takes pre-log curves in the support of $\gamma'$ to tropical curves in $\f{T}^{\ns}_{g,\Delta}(\A,\Psi)$.  Proposition~\ref{PreCount} tells us that for $\Gamma \in \f{T}^{\ns}_{g,\Delta}(\A,\Psi)$, $\gamma' \cap (f^{\trop})^{-1}(\Gamma)$ is the class of $\f{D}_{\Gamma}$ points, each with multiplicity $\frac{1}{|\Aut(\Gamma)|}$, and by Lemma \ref{NoAut} the pre-images of these points under $\forget_{\log}$ are not stacky.  Lemma~\ref{LogStructureCounts} tells us that the degree of $\forget_{\log}$ over a point with tropicalization $\Gamma$ is $\prod_{E\in \Gamma^{[1]}_c} w(E)$.  Hence, the degree of $\gamma^{\ns}$ is given by $\sum_{\Gamma\in \f{T}^{\ns}_{g,\Delta}(\A,\Psi)} \frac{\f{D}_{\Gamma}}{|\Aut(\Gamma)|} \prod_{E\in \Gamma^{[1]}_c} w(E)$, which by definition is ${}^{\ns}\GW_{g,N_{\bb{Q}},\Delta}^{\trop}(\A,\Psi)$, as desired.  Finally, the invariance of $\GW^{\log}_{g,\s{X}^{\dagger}_t,\Delta}(\A,\Psi)$ in the second statement is just Proposition \ref{GWInv}.
\end{proof}

\begin{rmk}\label{NoVertices}
In Definition \ref{TropCurveDfn}, we made the assumption that $\Gamma^{[0]}\neq \emptyset$.  We deal with this case here.  If we do have $\Gamma^{[0]}=\emptyset$, then $\Gamma$ necessarily consists of a single edge $E$ which is open at both ends.  The degree $[\Delta]$ is necessarily the class of a line.  As in Lemma \ref{NoMarks}(2), the corresponding log curve $C$ is uniquely determined by where it intersects the boundary.  The only conditions we can impose in these cases are boundary conditions $A_i$, $i=1,2$, and so the number of pre-log curves satisfying the conditions is just the number of points cut out by the $Z_{A_i,u_i}$'s.  That is, the number of pre-log curves corresponding to some $\Gamma$ with $\Gamma^{[0]}=\emptyset$ and satisfying conditions corresponding to rigid $A_1$ and $A_2$ is the index of $N/\bb{Z}u_E \rar N/(A_1\cap N) \oplus N/(A_2\cap N)$.  Rigid here means that $A_1$ and $A_2$ are transverse and their intersection is a line parallel to $\bb{Q} u_E$.  One easily checks that the virtual and actual dimensions agree in these cases and that these pre-log curves admit a unique basic log structure. At the end of the proof of Lemma \ref{NoAut}, we saw that these resulting log curves admit $w(E)$ distinct automorphisms, and we must divide by this to account for stackiness.  We thus find that \[\GW^{\log}_{0,\s{X}^{\dagger}_t,\Delta}(\A,\Psi) = \frac{1}{w}\inde(N/\bb{Z}u_E \rar N/(A_1\cap N) \oplus N/(A_2\cap N)),
\]
where $[\Delta]$ is $w$ times the class of a line, $\A=(A_1,A_2)$, and $\Psi=\emptyset$.

We note that it is possible to modify our framework to treat these $\Gamma^{[0]}=0$ cases simultaneously with the other cases by allowing unmarked bivalent vertices, saying that two tropical curves are equivalent if they can be related by adding and removing such vertices.  One then represents every tropical curve by a $\Gamma$ with $\Gamma^{[0]}\neq \emptyset$.  This results in minor changes through the proof, e.g., in the domain of $\Phi$, unmarked bivalent vertices $V$ correspond to factors of the form $N/\bb{Z}u_E$ ($E$ an edge containing $V$), and when defining $\Mult(\Gamma)$, in addition to multiplying by $w(E)$ for each compact edge $E$, one also divides by $w(E)$ for $E\ni V$ once for every unmarked bivalent vertex $V$.
\end{rmk}

\subsection{More general incidence conditions}\label{GenInc}
So far, we have assumed that our tropical incidence conditions correspond to affine linear spaces $A_i$.  However, most algebraic cycles do not have affine linear tropicalizations.  We describe here how to generalize to allow for arbitrary algebraic cycles. 

In Definition \ref{Constraints}, if $i\in I^{\circ}$, then in place of codimension $a_i$ affine subspace $A_i\subseteq N_{\bb{Q}}$, allow $A_i$ to be any tropical polyhedral complex in the sense of \cite{AR}.  For $i\in I^{\partial}$, in place of the codimension $a_i$ affine subspaces $A_i\subseteq N_{\bb{Q}}$, we consider codimension $a_i$ tropical cycles in $N_{\bb{Q}}/\bb{Q}u_{i}$.  As before, our tropical incidence conditions on interior points are the requirements $h(E_i)\in A_i$.  For the boundary points, we now require $\?{h(E_i)}\in A_i$, where the bar indicates that we have taken $h(E_i)+\bb{Q}u_i$ modulo $\bb{Q} u_{i}$.  Equivalently, we require that $h(E_i)$ is in the preimage of $A_i$ under the projection $N_{\bb{Q}}\rar N_{\bb{Q}}/\bb{Q}u_{i}$.

Note that the support of each $A_i$ is contained in a union of affine subspaces.  Hence, $\f{T}_{g,\Delta}(\A,\Psi)$ is contained in a union of spaces $\f{T}_{g,\Delta}(\A',\Psi)$ for $\A'$ an affine constraint as before.  We easily see by dimension counts that for generic elements of  $\f{T}_{g,\Delta}(\A,\Psi)$, the $h(E_i)$'s and $\?{h(E_i)}$'s are not contained in the higher-codimension strata of the $A_i$'s.  Thus, for a given generic element of $\f{T}_{g,\Delta}(\A,\Psi)$, we can define $A_{E_i}\subseteq N_{\bb{Q}}$ as the codimension $a_i$ affine space obtained by extending the cell of $A_i$ containing $h(E_i)$, and then define $\LL(E_i):=\LL(A_{E_i})$.  Similarly, for $i\in I^{\partial}$, we define $A_{E_i}\subseteq N_{\bb{Q}}/\bb{Q}u_{E_i}$ to be the affine subspace obtained by extending the cell of $A_j$ containing $\?{h(E_i)}$, and define $\LL(E_i):=\LL(B_{E_i})$.

With this setup, Proposition \ref{TropDeform} immediately generalizes to our situation, except that now we must restrict to some neighborhood of $(\Gamma,\epsilon,h)$ in the space of tropical curves in $\f{T}_{g,\Delta}(\A,\Psi)$ of the same combinatorial type as $(\Gamma,\epsilon,h)$, and this gets identified with a nonempty open subset of an open convex polyhedron in $\ker(\Phi_{\bb{Q}})$.  This modification is to ensure that the cells of the $A_i$'s which the tropical curves hit do not change.

Definition \ref{GeneralDfn} and Lemma \ref{General} also generalize easily, as does the map $\Phi$ from Equation \ref{D} and the fact that it has finite index.  We denote $\f{D}_{\Gamma}:= \inde(\Phi) \prod_{V\in \Gamma^{[0]}} \langle V \rangle$ as in \eqref{DGamma}.  We now define our multiplicities as follows:
\begin{dfn}
Fix some $(\Gamma,\epsilon,h)$ in a rigid $\f{T}^{\ns}_{g,\Delta}(\A,\Psi)$.  Denote by $w(A_{E_i})$ the weight of the cell of $A_i$ containing $h(E_i)$ (or $\?{h(E_i)}$ if $i\in I^{\partial}$).  Then define:
\begin{align}\label{GeneralMultGamma}
\Mult(\Gamma):=\frac{\f{D}_{\Gamma}}{|\Aut(\Gamma)|}\left(\prod_{E\in \Gamma^{[1]}_c} w(E)\right)\left(\prod_{i\in I} w(A_{E_i})\right).
\end{align}
${}^{\ns}\GW_{g,N_{\bb{Q}},\Delta}^{\trop}(\A,\Psi)$ is then defined as before, but using this generalized version of multiplicity.
\end{dfn}

On the algebraic side, we easily generalize the existence of a good polyhedral decomposition $\s{P}$ by again viewing the $A_i$'s and $B_i$'s as contained in unions of affine subspaces.  We can no longer guarantee that our tropical cycles $A_i$ are the tropicalizations of algebraic cycles $Z_{A_i}$ and $Z_{A_i,u_i}$ of $\s{X}$, but if they are, the generalization of $\GW^{\log}_{g,\s{X}^{\dagger}_t,\Delta}(\A,\Psi)$ and ${}^{\ns}\GW^{\log}_{g,\s{X}^{\dagger}_t,\Delta}(\A,\Psi)$ is obvious.

When the tropical cycles are realizable (as tropicalizations of algebraic cycles), then generically, we must have $[Z_{A_i}].[\s{X}_{\partial E_i}]=w(A_{\partial E_i})[Z_{A_{\partial E_i}}\cap \s{X}_{\partial E_i}]$, and similarly, $[Z_{A_i,u_i}].[\s{X}_{E_i}] = w(A_{E_i})[Z_{A_{E_i},u_{i}}\cap \s{X}_{E_i}]$.  The methods of \S \ref{tropicalization}-\ref{Trop2Log} now apply. 
Most importantly, while the intersection of $Z_{A_i}$ with a component of $\s{X}_0$ in general won't always be an orbit closure of a subtorus, it will be a subtorus orbit closure in the components where the incidence conditions are achieved, and similarly for the $Z_{A_i,u_i}$.  We thus find: 
\begin{thm}\label{MainThmExtended}
Theorem \ref{MainThm} holds for $\A$ consisting of arbitrary tropicalizations of algebraic cycles.
\end{thm}

\section{Applications: Ordinary GW invariants, Hurwitz numbers, and mirror symmetry}
\label{sec-applications}

\subsection{Ordinary descendant Gromov-Witten numbers}\label{OrdinaryGW}

Our interest so far has been in invariants $\GW^{\log}_{g,X^{\dagger},\Delta}(\A,\Psi)$, where $X^{\dagger}$ has the divisorial log structure with respect to its toric boundary. 
For $g=0$ and no boundary incidence conditions, we would now like to relate these to the ordinary (i.e., non-log) invariants $$\GW_{0,X,\Delta}(\A,\Psi):=\int_{[\s{M}_{0,I^{\circ}}(X,[\Delta])]^{\vir}} \left(\prod_{i\in I^{\circ}} \left(\psi_i^{s_i} \cup \ev_{x_i}^*[Z_{A_i}] \right)\right).$$
While the following discussion starts off general, the results we prove will be for when $X$ is a product of projective spaces. Set $e=e_\infty$.
The moduli space of basic stable log maps to a target $X$ with trivial log structure gives a log stack that becomes the stack of ordinary stable maps when forgetting its log structure. Furthermore, its log structure is simply the pull-back from the moduli space of pre-stable curves (the ghost sheaf stalk is $\bb{N}^{\#\hbox{\tiny nodes}}$). Gromov-Witten invariants for the trivial log structure on $X$ coincide with the ordinary Gromov-Witten invariants.
Composing a basic stable log map $C^\dagger\ra X^\dagger$ with the forgetful map $X^\dagger\ra X$ defines a functor 
$\shM_{g}(X^\dagger,\Delta)\ra \shM_{g,I}(X,[\Delta])$.
In view of \eqref{def-Q}, it is given at the level of ghost sheaf stalks by the composition
$\prod_q\bb{N} \hookrightarrow \wt{Q}\ra Q$. 
Recall that $\shM^{tt}_{g}(X^\dagger,\Delta)\subset \shM_{g}(X^\dagger,\Delta)$ denotes the open substack of non-singular torically transverse curves.  
We call $\Delta$ {\bf primitive} if all unbounded edges have weight $1$. Let $G=\Aut(\Delta)$ denote the finite group that permutes marking labels of the same type, meaning markings mapping to the same boundary divisor that have the same order of tangency (alias weights).  
That is, for $\Sym(S)$ denoting the group of permutations of a set $S$, 
 \begin{equation}\label{eq-aut}
\Aut(\Delta)=\prod_{{u\in N\setminus \{0\}}} \Sym(\{i\in I | \Delta(i)=u\}).   
\end{equation}
We have a commutative diagram
\begin{equation} 
\label{log-forget-diagram}
\xymatrix{
\shM^{tt}_g(X^{\dagger},\Delta)
\ar_{\mathrm{\tiny open}}@{^{(}->}[d]\ar^{\sim}[r]& 
\shM^{D,tt}_{g,I}(X,[\Delta])
\ar_{\mathrm{\tiny open}}@{^{(}->}[d]\ar[rrd]^{G-\hbox{\tiny torsor}}& 
\\ 
\shM_g(X^{\dagger},\Delta) \ar[r]&
\shM^D_{g,I}(X,[\Delta]) \ar[d]\ar@{^{(}->}[r]&
\shM_{g,I}(X,[\Delta]) \ar[r]\ar^{(\op{ev}_{i})_{i\in I^{\partial}}}[d] & \shM_{g,I^{\circ}}(X,[\Delta])
\\
&\prod_{i\in I^{\partial}} D_i \ar@{^{(}->}[r]&{\prod_{i\in I^{\partial}} X}
}
\end{equation}
where the bottom square is Cartesian and defines its top left corner $\shM^D_{g,I}(X,[\Delta])$.
The log-forget functor 
$\shM_{g}(X^\dagger,\Delta)\ra \shM_{g,I}(X,[\Delta])$ (from the middle row)
factors through $\shM^D_{g,I}(X,[\Delta])$ because the boundary markings $i\in I^{\partial}$ are automatically mapping to the corresponding toric boundary divisors $D_i$. The right-most horizontal map is the map that forgets these boundary markings.  Note that the log structure on $\shM^{tt}_{g}(X^\dagger,\Delta)$ is trivial and the one on its universal curve is the pull-back from $X^\dagger$, so in particular $\shM^{tt}_{g}(X^\dagger,\Delta)$ maps isomorphically onto its image in $\shM^D_{g,I}(X,[\Delta])$, which we denote
$\shM^{D,tt}_{g,I}(X,[\Delta])$ giving the top left Cartesian square.  Finally, the diagonal arrow is indeed a $G$-torsor (onto a subset of the target that is open if $\Delta$ is primitive) 
because of the $|G|$ choices to recover markings after forgetting them.
Note that all spaces except for those in the top row are proper.
By absence of automorphisms, the two spaces in the top row are actually schemes.
If all of our $\psi$-classes are associated with point conditions (as is the case for the situations considered in \cite{MR} and \cite{Rau}), then the correspondence theorem for products of projective spaces reads as follow.

\begin{thm}
\label{thm-forgetlog-equal}
We assume that $X=\prod_{i=1}^k \bb{P}^{r_i}$ and that $\Delta$ is primitive.
Suppose there are no $\psi$-conditions, i.e. $s_i=0$ for each $i$, or more generally that each $\psi$-condition is attached to a point condition, or yet more generally that each $\psi$-condition is attached to an incidence condition whose intersection with  the boundary component $D_i$ is $0$ whenever $[\Delta].D_i\neq 0$.  Then 
$$\GW^{\log}_{0,X^{\dagger},\Delta}(\A,\Psi)
=
|G|\cdot\GW_{0,X,[\Delta]}(\A,\Psi).$$
\end{thm}

\begin{proof} 
Since $X$ is convex, the virtual fundamental classes of all moduli spaces in 
\eqref{log-forget-diagram}
are the usual fundamental classes (by \cite[Prop.  3.3.6]{Ran} $\shM^{tt}_{0}(X^\dagger,\Delta)$ is dense in $\shM_{0}(X^\dagger,\Delta)$).
We focus on the map 
$F:\shM_{0}(X^\dagger,\Delta)\ra \shM_{0}(X,[\Delta])$ that forgets log structure and boundary markings. The source and target have the same dimension, so by means of this map, $\shM^{tt}_{0}(X^\dagger,\Delta)$ is a $G$-torsor over a dense open subset of the target.
Note that the assertion follows if we prove that the zero-cycle giving the incidence and $\psi$-conditions on 
$\shM_{0}(X,[\Delta])$ is contained in this open set in the target of $F$ and pulls back under $F$ to the analogously defined zero-cycle in $\shM_{0}(X^\dagger,\Delta)$. 
Each incidence condition is defined by intersecting general hyperplanes in the projective space factors of $X$.
They clearly pull back under $F$ and by Bertini-Kleiman's theorem intersect in the open subset $F(\shM_{0}(X^\dagger,\Delta))$, respectively in $\shM_{0}^{tt}(X^\dagger,\Delta)$, and we are done if we can make a similar statement about the $\psi$-classes.
Note that the difference between $\psi$-classes on source and target of $F$ is given by stable maps that have the $\psi$-marking on a domain component $C_0$ that destabilizes under $F$. 
We claim such stable maps do not exists in the locus of curves that satisfy the incidence condition of this marking. 
If $C_0$ maps non-constantly into $X$ before applying $F$, it will also do so after applying $F$, so this case is clear. If $C_0$ is contracted, then it is contracted to the incidence condition associated to the marking of the $\psi$-class, and by the $0$-intersection assumption this will generically be disjoint from the boundary divisors to which the boundary marked points map.  Hence, forgetting boundary marked points will not affect the stability of $C_0$, so the $\psi$-classes and pullback $\psi$-classes will indeed agree locally wherever the incidence conditions are satisfied.

It remains to say that the $\psi$-conditions are given by basepoint free divisors:
by Prop.~\ref{psihat}, the $\psi$-classes pull back from the underlying moduli space of curves, and by Lemma~\ref{lem-psi-basepointfree} these are base point-free.
\end{proof}

\begin{rmk}
Note that a zero-cycle pulled back from $\shM_{0,I^{\circ}}(X,[\Delta])$ as used in the proof is automatically $G$-invariant. However, the zero-cycle obtained in the proof of the main Theorem~\ref{MainThm} (comparing log and tropical counts) is generally \emph{not} $G$-invariant if $\psi$-conditions are present. Indeed, permuting boundary markings on a stable map domain component that also has a $\psi$-marking generally doesn't preserve the $\psi$-condition as we represented it by fixing the underlying curve component (e.g. the cross-ratio of $4$ points on $\bb{P}^1$ isn't preserved when changing the point order).
\end{rmk}

When we have $\psi$-conditions attached to more general cycles than those allowed in Thm~\ref{thm-forgetlog-equal}, 
it is no longer true that the Gromov-Witten zero cycle on $\shM_{0,I^{\circ}}(X,[\Delta])$ pulls back to the corresponding one on
$\shM_{0}(X^\dagger,\Delta)$. 
We expect a more sophisticated relationship between the log- and ordinary invariants that needs to deal with the fact that $\shM^D_{0,I}(X,[\Delta])$ in general is not irreducible, in fact not even pure-dimensional. Log curves map into only one of its components despite there possibly being additional curves in other components which still satisfy incidence and $\psi$-conditions. 
We include here an easy case where this relationship reduces to the ordinary divisor equation.
We first record a lemma. 
Since $\shM_{0}(X^\dagger,\Delta)\ra \shM_{0}(X,[\Delta])$ preserves the underlying pre-stable curve (the universal curve pulls back under this map), we have the following result.

\begin{lem} \label{forgetlog-psipullback}
The $\psi$-classes on $\shM_{0}(X,[\Delta])$ pull back to the $\psi$-classes on $\shM_{0}(X^\dagger,\Delta)$.
\end{lem}

\begin{thm} \label{main-comparison-thm}
Assume $X=\prod_{i=1}^k \bb{P}^{r_i}$,
and $[D].[\Delta]\le 1$ for every prime boundary divisor $D\subset X$.  Then for every set of incidence and $\psi$-conditions $(\A,\Psi)$, the log Gromov-Witten invariant coincides with the usual one under replacing boundary markings 
with divisor incidence conditions.  In symbols, letting $D_j$ denote the boundary divisor to which the boundary marked point $y_j$ maps, we have:
$$\GW^{\log}_{0,X^{\dagger},\Delta}(\A,\Psi)
=
\GW_{0,X,[\Delta]}((\A,(D_i)_{i\in I^{\partial}}),\Psi).$$
\end{thm}

\begin{proof} 
In view of the bottom Cartesian square in \eqref{log-forget-diagram}, we can use Bertini-Kleiman's theorem to deduce that
$\shM^{D}_{0,I}(X,[\Delta])$ is smooth, proper and irreducible. Indeed this follows, because 
$\shM_{0,I}(X,[\Delta])$ is smooth, proper and irreducible and the $D_i$'s are in general position in the homogeneous space $X$ (using that each $[D].[\Delta]\leq 1$). 
Since $X$ is convex, the virtual fundamental classes of all moduli spaces in 
\eqref{log-forget-diagram}
are the usual fundamental classes.
Incidence and $\psi$-conditions pull back (as generalized Gysin maps) from $\shM_{g,I}(X,[\Delta])$ by Lemma~\ref{forgetlog-psipullback}.

As in the proof of the previous theorem: by Prop.~\ref{psihat}, the $\psi$-classes pull back from the underlying moduli space of curves, and by Lemma~\ref{lem-psi-basepointfree} these are base-point free.  Also, the incidence conditions (being multiples of intersections of hyperplane sections) are base-point free. Hence, by another application of Bertini, the cycle
$\gamma=\psi_1^{s_1}\cap [Z_{A_1}]\cap...\cap \psi_m^{s_m}\cap [Z_{A_m}]\cap[\shM^D_{0,I}(X,[\Delta])]$
is contained in the dense open set $\shM^{D,tt}_{0,I}(X,[\Delta])$. The isomorphism to $\shM^{tt}_{0}(X^\dagger,\Delta)$ and compatibility of 
$\gamma$ with pull-back and the degree map yields the assertion.
\end{proof}

\begin{eg} \label{example-double-psi}
Let us compute the log Gromov-Witten invariant of lines in $\bb{P}^2$ satisfying a $\psi^1$-condition at a point and a $\psi^1$-condition at the fundamental class. There is a unique tropical line satisfying these conditions. It looks like the line on the right in Fig.~\ref{trop-line} and has both markings in the vertex $V$. Its multiplicity is $\langle V\rangle = \frac{\ov(V)!}{s_1!s_2!}=2$ which is thus also the log Gromov-Witten invariant. 
Let $D_1,D_2,D_3$ be the toric boundary components in $\bb{P}^2$ and $\ell$ the class of a line.  Then
we can verify this log invariant by using Thm.~\ref{main-comparison-thm} as follows:
\begin{align*}
\langle\psi[pt],\psi[\bb{P}^2]\rangle_{0,\ell}^{\bb{P}^2(\log (D_1+D_2+D_3))} 
&\underset{\hbox{\tiny Thm.~\ref{main-comparison-thm}}}{=} \langle\psi[pt],\psi[\bb{P}^2],D_1,D_2,D_3\rangle_{0,\ell}^{\bb{P}^2}\\
&\underset{\hbox{\tiny dil.eq.}}{=}
(4-2)\langle\psi[pt],D_1,D_2,D_3\rangle_{0,\ell}^{\bb{P}^2}
\underset{\hbox{\tiny (div.eq.$)^3$}}{=}
2\underbrace{\langle\psi[pt]\rangle_{0,\ell}^{\bb{P}^2}}_{=1}
\end{align*}
where the second equality is the dilaton equation \cite[\S4.3.1]{Kock} and the third is the divisor equation \cite[\S4.3.2]{Kock}. Note this differs from the ordinary invariant $\langle\psi[pt],\psi[\bb{P}^2]\rangle_{0,\ell}^{\bb{P}^2}=-1$.
\end{eg}

We end this section by introducing a tropical multiplicity giving the ordinary Gromov-Witten invariant in the situations covered by Theorem~\ref{thm-forgetlog-equal}.

Let $\f{T}'_{g,\Delta}(\A,\Psi)$ denote the modification of $\f{T}_{g,\Delta}(\A,\Psi)$ obtained by forgetting all the boundary markings---i.e., identifying curves related by the boundary relabeling group $\Aut(\Delta)$. 
Let $\Stab_{\Gamma}(\Delta)$ denote the subgroup of $\Aut(\Delta)$ acting trivially on $\Gamma$. 
Let $\rho_1,\ldots,\rho_s$ denote the rays of the fan corresponding to $X$.  For each vertex $V$ of a tropical curve $\Gamma$, let $n^w_i(V)$ denote the number of unbounded rays of weight $w$ emanating from $V$ in the direction of $\rho_i$.  By Lemma~\ref{VerticesDistinct}, distinct vertices of $\Gamma$ cannot map to the same point of $N_{\QQ}$, so any automorphism of $\Gamma$ must fix $\Gamma^{[0]}$.  It follows that any element of $\Aut(\Delta)$ stabilizing $\Gamma$ can only act by permutations of the sets of $n^w_i(V)$ edges for each $V$, $i$, and w.  We have
\[|\Stab_{\Gamma}(\Delta)| = \prod_{V,i}\prod_{w\ge 1} n^w_i(V)!.\]
For $\Gamma \in  \f{T}_{g,\Delta}(\A,\Psi)$ and $\Gamma'$ the corresponding curve in $\f{T}'_{g,\Delta}(\A,\Psi)$, define $\Mult(\Gamma')$ exactly as we have done for $\Mult(\Gamma)$, but note that $\Aut(\Gamma')$ has changed because automorphisms no longer have to respect boundary markings.  Hence,
\begin{align*}
    \Mult(\Gamma')=\frac{\Mult(\Gamma)}{|\Stab_{\Gamma}(\Delta)|}.
\end{align*}

Define
\begin{align}\label{def-usual-tropical}
\GW_{0,N_{\bb{Q}},\Delta}^{\trop'}(\A,\Psi)&:=\sum_{\Gamma'\in \f{T}'_{0,\Delta}(\A,\Psi)} \Mult(\Gamma').
\end{align}

Note that $\Gamma' \in \f{T}'_{0,\Delta}(\A,\Psi)$ has $\frac{|\Aut(\Delta)|}{\Stab_{\Gamma}(\Delta)}$ lifts to $\f{T}_{0,m,\Delta}(\A,\Psi)$.  It follows that 
\[
\GW_{0,N_{\bb{Q}},\Delta}^{\trop'}(\A,\Psi)= \GW_{0,N_{\bb{Q}},\Delta}^{\trop}(\A,\Psi)/|\Aut(\Delta)|.
\]

This and Theorem \ref{thm-forgetlog-equal} imply the following corollary:
\begin{cor} \label{tropical-matches-usual}
Suppose $g=0$, $X=\bb{P}^{r_1}\times \cdots \times \bb{P}^{r_k}$ and $\Delta$ primitive. If all $\psi$-classes are attached at point conditions, i.e. $\dim A_i = 0$ whenever $s_i >0$, (or more generally if each $\psi$-condition is attached to an incidence condition whose intersection with $[\Delta]$ in the Chow ring is $0$) then 
\begin{align*}\GW_{0,X,\Delta}(\A,\Psi) =\GW^{\log}_{0,X^{\dagger},\Delta}(\A,\Psi)/|\Aut(\Delta)| 
                                            =\GW_{0,N_{\bb{Q}},\Delta}^{\trop'}(\A,\Psi).\end{align*}
\end{cor}

\begin{eg}\label{Boundary-Markings}
Let us compute the log Gromov-Witten invariant\\ 
\begin{minipage}[b]{0.60\textwidth} 
$$\langle[pt]\psi,[pt]\psi,[pt]\psi,[pt]\psi\rangle_{0,3[\bb{P}^1]}^{\bb{P}^1(\log(0\cup\infty))}=16\cdot 3\cdot 3$$
of degree three maps $\bb{P}^1\ra\bb{P}^1$ that are unramified over $0$ and $\infty$ with four single $\psi$-classes at point conditions. Precisely the two tropical curves depicted on the right contribute to this
count. Indeed, $\Mult(\Gamma_1)=4$ and $\Mult(\Gamma_2)=12$ are just 
\end{minipage}\qquad
\begin{minipage}[t]{0.40\textwidth}
\resizebox{0.9\textwidth}{!}{
\includegraphics{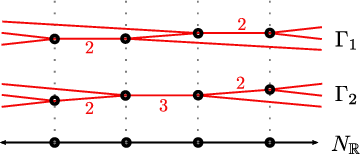}
}
\end{minipage}\\
products of weights and there are $3\cdot 3$ of each of these when considering different labelings of the unbounded edges (note that $|\Stab_{\Gamma}(\Delta)|=4$).  Since $|\Aut(\Delta)|=(3!)^2$, from Theorem~\ref{thm-forgetlog-equal} we obtain for the ordinary Gromov-Witten invariant
$$\langle[pt]\psi,[pt]\psi,[pt]\psi,[pt]\psi\rangle_{0,3[\bb{P}^1]}^{\bb{P}^1}=4.$$
\end{eg}

\subsection{Tropical Hurwitz numbers}  \label{section-hurwitz}
As an example in dimension one, we next interpret Hurwitz numbers \cite{Hu} for target $\bb{P}^1$ in terms of descendant log/tropical Gromov-Witten counts similar to \cite{CJM,CJMR}. A descendant Gromov-Witten theoretic interpretation for cases not ramified over $0$ and $\infty$ was previously given in \cite{Pan2}.

\begin{dfn}
\label{def-Hurwitz}
Let $\alpha=[a_1+\ldots+a_{\ell_1}=d]$ and $\beta=[b_1+\ldots+b_{\ell_2}=d]$ be two ordered partitions of $d>0$.  The {\bf double Hurwitz number} $H_{g,d}(m,\alpha,\beta)$ is defined as follows: Let $\{p_1,\ldots,p_m\}$ be a generically specified collection of points in $\bb{P}^1$.  Then $H_{g,d}(m,\alpha,\beta)$ is defined to be the number of degree $d$ marked covers $\pi:(C,y_1^0,\ldots,y_{\ell_1}^0,y_1^{\infty},\ldots,y_{\ell_2}^{\infty})\rar \bb{P}^1$, up to isomorphism, such that:
\begin{itemize}
\item $C$ is a smooth connected curve of genus $g$,
\item $\pi$ is unramified over $\bb{P}^1 \setminus \{0,\infty,p_1,\ldots,p_m\}$,
\item $\pi$ is simply ramified over $p_i$ for each $i=1,\ldots,m$.
\item For each $i=1,\ldots,\ell_1$, $\pi$ is ramified over $0$ at $y_i^0$ to order $a_i$.  Similarly, for each $j=1,\ldots,\ell_2$, $\pi$ is ramified over $\infty$ at $y^{\infty}_j$ to order $b_j$.
\end{itemize} 
The count is weighted by $\frac1{\Aut(\pi)}$.
Note that $H_{g,d}(m,\alpha,\beta)\neq 0$ implies by the Riemann-Hurwitz theorem that
\begin{equation}\label{RH-m}
m = 2g-2+\ell_1+\ell_2,
\end{equation} 
so we let the notation $H_{g,d}(\alpha,\beta)$ (or even just $H_{g,d}$) refer to $H_{g,d}(m,\alpha,\beta)$ for $m$ given by \eqref{RH-m}.
\end{dfn}

Hurwitz \cite{Hu} originally studied these numbers by factorizations in the symmetric group $S_d$. Indeed, let
$\hat H_{g,d}$ denote the number of solutions to
\begin{equation} \label{eq-Hurwitz}
\sigma_\alpha\tau_1...\tau_m\sigma_\beta=1
\end{equation}
by $\sigma_\alpha,\sigma_\beta,\tau_1,...,\tau_m\in S_d$ where $\tau_i$ are transpositions, $\sigma_\alpha,\sigma_\beta$ have cycle type $\alpha, \beta$ respectively and all these together act transitively on $\{1,...,d\}$. The relationship to Def.\,\ref{def-Hurwitz} is that a cover $C\ra\bb{P}^1$ is uniquely determined by its monodromy representation and such is given by \eqref{eq-Hurwitz}, cf. \cite[\S1]{Hu}.
Note that for the count $\hat H_{g,d}$ all $d$ sheets are labeled whereas for $H_{g,d}$ instead the points over $0$ and $\infty$ are labeled. 
Define a third count $\bar H_{g,d}$ with no labeling at all, i.e. by the analogue of Def.\ref{def-Hurwitz} without the marking by $y_i^0, y_j^\infty$.
Let $\Aut(\alpha)$ be the group of permutations of the labels $y_i^0$ that preserve branch order, so
$|\Aut(\alpha)|=\prod_i \mathfrak{a}_i!$
where $\mathfrak{a}_i$ denotes the number of summands in $\alpha$ equal to $i$. Define $\Aut(\beta)$ similarly.
That we divide all counts by automorphisms respecting the appropriate labelings enables the following clean comparison result.

\begin{lem} \label{lem-label-versions}
$\bar H_{g,d}\ =\ \frac1{d!}\hat H_{g,d}$ and
$H_{g,d}=\frac{|\Aut(\alpha)|\cdot |\Aut(\beta)|}{d!} \hat H_{g,d}$.
\end{lem}
\begin{proof}
The second equality is \cite[Prop.\,1.1]{GJV05} and the first works similar, namely by the sheet relabeling action of $S_d$ via conjugation on the solutions to \eqref{eq-Hurwitz}: if a solution happens to be conjugated to itself, this permutation actually gives an automorphism of the cover that hence is also divided out by definition.
\end{proof}

\begin{dfn}
Let $\alpha$ and $\beta$ be two ordered partitions of $d>0$.  Let $n=1$, so $N\cong \bb{Z}$.  Let $\Delta=\Delta(\alpha,\beta)$ be the degree associated to $\alpha$ and $\beta$ (so tropical curves of degree $\Delta$ have $\ell_1$ marked unbounded edges in the negative direction of weights $a_1,\ldots,a_{\ell_1}$, and $\ell_2$ marked unbounded edges in the positive direction of weights $b_1,\ldots,b_{\ell_2}$). 
Let $\A$ consist of $m:=2g-2+\ell_1+\ell_2$ generic interior point conditions, no boundary conditions and $\Psi=(1,\ldots,1)$ ($m$ entries, all equal to $1$). 
It is not hard to see that the corresponding $\f{T}^{\ns}_{g,\Delta}(\A,\Psi)$ is rigid.  The {\bf tropical double Hurwitz number} is then defined as
\begin{equation} \label{def-Htrop}
H_{g,d}^{\trop}(\alpha,\beta):=\GW^{\trop}_{g,\Delta}(\A,\Psi)
\end{equation}
Using the notation of \eqref{def-usual-tropical}, we similarly define 
$$\bar H_{g,d}^{\trop}(\alpha,\beta):=\GW^{\trop'}_{g,\Delta}(\A,\Psi)= H_{g,d}^{\trop}(\alpha,\beta)/|\Aut(\Delta)|$$
\end{dfn}

We note that $\f{D}_{\Gamma}$ always equals $1$ for $n=1$, so for $\Gamma \in \f{T}^{\ns}_{g,\Delta}(\A,\Psi)$ as above, we always have $\Mult(\Gamma)= \frac{\prod_{E\in \Gamma^{[1]}_c} w(E)}{|\Aut(\Gamma)|}$. Cavalieri, Johnson, and Markwig prove the following.

\begin{thm}[\cite{CJM}] \label{thm-CJM}
$\bar H_{g,d}(\alpha,\beta) = \bar H_{g,d}^{\trop}(\alpha,\beta)$.
\end{thm}

\begin{cor} \label{cor-CJM}
$H_{g,d}(\alpha,\beta) =  H_{g,d}^{\trop}(\alpha,\beta)$.
\end{cor}
\begin{proof}
This follows immediately from Theorem \ref{thm-CJM} combined with Lemma \ref{lem-label-versions} and the description of $\Aut(\Delta)$ in \eqref{eq-aut}.
\end{proof}

This and our Main Theorem \ref{MainThm} immediately imply the following:
\begin{thm}\label{H-GW}
$H_{g,d}(\alpha,\beta) =\GW^{\log}_{g,(\bb{P}^1)^{\dagger},\Delta}(\A,\Psi)$. 
\end{thm}

We note that this relationship was also recently observed in \cite{CJMR}.  In particular, the one-dimensional case of our correspondence theorem with all incidence conditions being point conditions is essentially their Theorem 3.1.2.

For a better algebro-geometric understanding, one may prove Theorem \ref{H-GW} directly, thus reproving Theorem \ref{thm-CJM} as a consequence of our main theorem. The case with no ramification over $0$ and $\infty$ is \cite[Prop. 3]{Pan2}.  The argument used there works identically for our setup (including the unobstructedness argument from the proof of \cite[Prop. 2]{FP}).  We note that the argument shows that the correspondence is explicit---i.e., for certain special representatives of the $\psi$-classes, the curves cut out by $\gamma$ are precisely the curves counted by the Hurwitz numbers. The main idea is as follows: for $s$ a meromorphic section of $T^*\bb{P}^1$ with simple poles at $0$ and $\infty$ and no zeroes or poles anywhere else, $\ev_{x_i}^*s$ is a meromorphic section of the $\psi_{i}$-class line bundle.  Furthermore, a smooth curve is cut out by $\ev_{x_i}^*s\cap\ev_{x_i}^*[\pt]$ if and only if it is ramified at the point $x_i$ mapping to the generically specified representative of $[\pt]$.  That is, satisfying the point and $\psi$-class conditions is equivalent to satisfying the interior simple ramification conditions.  Of course, the ramification conditions at the boundary are satisfied for all log curves with smooth domain and degree $\Delta$.

\begin{rmk}
We note that the no-automorphism-Lemma \ref{NoAut} does not hold in these non-general situations.  In the map $\wt{\s{M}}(\Gamma)\rar \s{M}(\Gamma):= \wt{\s{M}}(\Gamma)/\Aut(\Gamma)$, points satisfying these special conditions are where the map is maximally ramified, i.e., where it looks like $\Spec \kk \rar \Spec \kk /\Aut(\Gamma)$, as opposed to the generic fiber $(\Spec \kk) \times \Aut(\Gamma) \rar \Spec \kk$.  So the special conditions cut out points in the central fiber with multiplicity $1$ but stabilizer $\Aut(\Gamma)$, as opposed to non-stacky points with multiplicity $\frac{1}{|\Aut(\Gamma)|}$ as happens generically.  Of course, this does not affect the contribution to the Gromov-Witten numbers as the two versions are rationally equivalent as zero-cycles.  
\end{rmk}

\begin{eg} \label{ex-hyperell}
Consider the tropical curve $\Gamma= \xymatrix{
\ar^2@{-}[r]^(-.15){}="a"^(1.15){}="b" \ar@{-} "a";"b"
&\bullet
\ar@{=}[r]^(-.1){}="a"^(1.1){}="b" \ar@{=} "a";"b"
&\bullet \ar^2@{-}[r]^(-.15){}="a"^(1.15){}="b" \ar@{-} "a";"b"
&\bullet
\ar@{=}[r]^(-.1){}="a"^(1.1){}="b" \ar@{=} "a";"b"
&\bullet \ar^2@{-}[r]^(-.15){}="a"^(1.15){}="b" \ar@{-} "a";"b"
&
}$ for the target $\PP^1$.
One checks that this $\Gamma$ is the only element of the corresponding $\f{T}^{\ns}_{g,\Delta}(\A,\Psi)$. Since $\prod_{E\in \Gamma^{[1]}_c} w(E)=2$ and $\Aut(\Gamma)\cong(\bb{Z}/2\bb{Z})^2$, this gives $H^{\trop}_{2,2}(2,2)$ as 
$\Mult(\Gamma)=\frac{2}{2^2}=\frac12$. 
Note that this calculation reflects the situation with generic $\psi$-conditions as for the proof of Theorem~\ref{MainThm} where indeed there is a unique pre-log curve giving two distinct stable log maps \emph{without} automorphisms (due to Lemma~\ref{NoAut}) and
each contributes $\frac14$ to the log invariant.
However, with the special $\psi$-classes given by the section $s$ we get a unique stable log map with an automorphism group of order two reflecting the geometry of the Hurwitz cover. In particular, the $\psi$-conditions via $s$ are non-generic.
\end{eg} 

\begin{eg}
It follows from Example \ref{Boundary-Markings} that for $\alpha=[1+1+1]$ and $\beta=[1+1+1]$, $H_{0,3}(\alpha,\beta) = 16\cdot 3 \cdot 3$ and  $\bar H_{g,d}(\alpha,\beta)=4$.
\end{eg}

\begin{rmk}
We can easily modify our main theorem to work for elliptic curves, and then we recover the correspondence theorem of \cite{BBBM} as a corollary of that.  This is essentially \cite[Thm. 3.2.1]{CJMR}.  See also their Def. 3.2.2 for a similar setup allowing for even more general target curves.
\end{rmk}

\subsection{Applications motivated by mirror symmetry} 
\label{motivation-mirror}
Our correspondence theorem is particularly useful for the Gross-Siebert program, in which one approaches mirror symmetry by first describing the $A$- and $B$-models tropically.  In particular, the authors of \cite[{\bf arxiv version 1}, \S 0.4]{GHK1v1} made the ``Frobenius structure conjecture'' that states that the mirror dual to a log Calabi-Yau variety is obtained as $\Spec R$ for $R$ a certain ring with a canonical basis of so-called theta functions.  The conjectured multiplication rule for these basis elements is expressed in terms of certain descendant log Gromov-Witten invariants.  Tropically, the descendant conditions show up because a product of $s$ theta functions should involve gluing $s$-tuples of tropical disks together at a marked vertex, resulting in an $(s+1)$-valent vertex, hence curves satisfying a $\psi^{s-2}$-condition.  In \cite{Man3}, the first author showed that theta functions on cluster varieties as in \cite{GHKK} can indeed be expressed in terms of descendant tropical curve counts.  Then in \cite{ManFrob}, the first author uses this tropical description along with our descendant tropical correspondence theorem and degeneration techniques to verify the Frobenius structure conjecture for cluster varieties.

As another application, the calculation of the 2875 lines on a quintic threefold by means of degeneration involves counting straight lines in $\bb{P}^3$ that meet four quintic curves in the four boundary planes (and don't meet the coordinate lines). 
By our main theorem \ref{MainThm}, this count is tropical. Using the resulting tropical curves in \cite{MakRud}, Cheuk Yu Mak and the second author construct for each tropical line of multiplicity $p$ a Lagrangian graph manifold $L$ with $|H_1(L)| =p$ in the mirror dual quintic (to be thought of as the mirror dual object to the complex line).
By generalized Dehn-Seidel twists around these lens spaces, we produce a large Abelian subgroup of the symplectic automorphism group of the mirror quintic.

%=========================================================================================================================================

\appendix

\section{Log Gromov-Witten invariants are constant in log smooth families}
Gromov-Witten invariants are known to be constant in smooth families, but they are typically not constant under degenerations.  One of the main motivations behind the definition of log Gromov-Witten invariants was to obtain a theory which is invariant not only in smooth families, but also in log smooth families.  We give a proof of this invariance here.

We must first recall some facts about intersection theory on Deligne-Mumford stacks (DM stacks).  
If $f:F\ra G$ is a representable morphism of DM stacks, we obtain as in \cite[\S5]{Vistoli} bivariant Chow groups $A^\bullet(F\ra G)$, cf. \cite[\S17]{Fult2}. An element of this is by definition a collection of homomorphisms $A^\bullet(Y)\ra A^\bullet(F\times_G Y)$ one for each $G$-scheme $Y\ra G$ and this collection commutes with flat pullback, proper pushforward and Gysin maps. 
 It is explained in \cite[middle of p.652]{Vistoli} how by using a presentation and the degree formula, an element of $A^\bullet(F\ra G)$ naturally also gives homomorphisms $A^\bullet(Y)\ra A^\bullet(F\times_G Y)$ for all $Y$ that are a DM stack over $G$.
In particular by \cite[Lemma 5.3]{Vistoli}, we have
\begin{lem} 
For $i:F\ra G$ a regular embedding of DM stacks and $f:G'\ra G$ a proper morphism of DM stacks, we have for any $\alpha\in A_m(G')$,
$$f'_*i^!\alpha =  i^!f_*\alpha$$
for $f'$ the base change of $f$ to $F$.
\end{lem}

This Lemma is the key ingredient for \cite[Prop. 10.1(a)]{Fult2} which in turn is needed for \cite[Prop. 10.2]{Fult2}, so we obtain a generalization to DM stacks that looks as follows.
\begin{lem}
\label{lemma-number-constant-schemes}
Let $f:M\ra S$ be a proper morphism of DM stacks with $S$ regular and irreducible of dimension $d$, $\alpha\in A_d(M)$, 
$s:*\ra S$ a point, $M_s=f^{-1}(s(*))$ then the degree of $\alpha_s:=s^!\alpha\in A_0(M_t)$ is independent of $s$ and coincides with the coefficient of $\alpha$ at $[S]$.
\end{lem}

Now, as in the start of \S \ref{Moduli}, suppose we have a projective, log smooth morphism of fine saturated log schemes $f^{\dagger}:Y^{\dagger}\rar S^{\dagger}$, as well as a combinatorially finite conditions $\beta$.  By \cite[Thm 0.2 and 0.3]{GSlog}, the stack of basic stable log maps $M:=\s{M}(Y^{\dagger}/S^{\dagger},\beta)$ is a Deligne-Mumford stack that is proper over $S$ and comes with a natural virtual fundamental class $[M]^{\vir}\in A_{\vdim(M)+\dim S}(M)$. Let $\pi:M\ra S$ denote this proper map.
Now let $\gamma\in A^{\vdim(M)}(M)=A^{\vdim(M)}(M \ra M)$ be a bivariant class and assume $S$ is integral, so its fundamental class $[S]$ is well-defined.
We think of $\gamma$ as incidence and $\psi$-class conditions as in Def.\,\ref{def-log-GW} that cut down the moduli of stable maps to something virtually finite over $S$. 
We define the log Gromov-Witten invariants with respect to $\beta$ and $\gamma$ as
$$N_\beta(\gamma):=\hbox{coefficient of }\pi_*(\gamma\cap [M]^{\vir}) \hbox{ at }[S].$$
\begin{thm}\label{Invariance} 
Let $s:*\hookrightarrow S$ be a regular point and $Y_s:=Y\times_S s(*)$.  Let $\gamma_s:=s^!\gamma$. Then 
$$N_\beta(\gamma_s)=N_\beta(\gamma).$$
I.e. the log Gromov-Witten invariants of the fibers of $f$ with respect to $\gamma$ and $\beta$ are constant over regular points of $S$.
\end{thm}
\begin{proof} Equip $*$ and $Y_s$ with the pulled back log structures to define $*^{\dagger}$ and $Y^{\dagger}$.  The construction of the virtual fundamental class is compatible with base change, so $[\shM(Y^{\dagger}_s/*^{\dagger},\beta)]^{\vir}=s^![\shM(Y^{\dagger}/S^{\dagger},\beta)]^{\vir}$. The result now follows from Lemma~\ref{lemma-number-constant-schemes}.
\end{proof}

\section{Higher-genus vertices deform to superabundant genus zero vertices}
We prove that, among superabundant tropical curves which satisfy some rigid collection of conditions, those with higher-genus vertices are not isolated.
This justifies the fact that we call tropical curves with higher-genus vertices superabundant. We assume $n\ge 1$ as before.
\begin{prop}\label{No-gV}
Let $\Gamma$ be a tropicalization\footnote{We say ``a tropicalization'' here because a single basic stable log map has in general a parameter space worth of tropicalizations.} of a basic stable log map $\varphi$ to $\s{X}_0$ of genus $g$ with degree $\Delta$, in the support of the Gromov-Witten invariant cycle $\gamma$ associated to some generically specified algebraic incidence and $\psi$-class conditions corresponding to rigid $\A$ and $\Psi$. If $\Gamma$ contains any higher genus vertices, then  $\f{T}_{g,\Delta}(\A,\Psi)$ also contains a superabundant curve supported on the image of $\Gamma$ in $N_{\bb{R}}$ which has $g_V=0$ for all $V$.
\end{prop}
\noindent
\begin{minipage}[b]{0.72\textwidth} 
\proof Let $\varphi:C^\dagger\ra X^\dagger_0$ be a stable log map that tropicalizes to $\Gamma$. Let $\bar \Gamma $ denote the 
``collapse'' of $\Gamma$ where we replace every maximal connected subgraph $\Gamma'$ that maps to the same point in $N_\RR$ by a single vertex $V$, replacing all (resulting and previous) self-adjacent edges at $V$ with pairs of non-compact edges, while attaching at $V$ the edges that go from $\Gamma'$ to other parts of $\Gamma$, along with any half-edges which were already contained in $\Gamma'$.  We give $V$ the genus that reflects the genus of the part of the stable map $\varphi$ which it represents, i.e., $$b_1(\Gamma')+\sum_{V\in {\Gamma'}^{[0]}} g_V.$$
We also reattach all decorations in the obvious way to $\bar\Gamma$.
\end{minipage}\qquad
\begin{minipage}[t]{0.28\textwidth}
\resizebox{0.9\textwidth}{!}{
\includegraphics{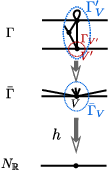}
}\\
\end{minipage}\\
For later reference, we denote the subgraph of $\Gamma$ that collapses to $V$ by $\Gamma'_V$ and we let $\bar\Gamma_V$ be the vertex $V$ with all adjacent edges in $\bar \Gamma$ and the induced decorations.

Being a tropicalization, $\Gamma$ has vertices in $N_{\bb{Q}}$ and so does $\bar \Gamma$. 
We may thus assume that $\bar\Gamma$ has been included in the $1$-skeleton of $\s{P}_{g,\Delta}(\A,\Psi)$ as in Lemma \ref{lem-good-subd}. 
Consequently, the nodes of the log curve that correspond to non-contracted edges of $\Gamma$ map into the purely codimension one strata, i.e. away from the complement of all strata of codimension $\ge 2$ in $\s{X}_0$. 
In \cite[\S5,\,\S6]{KLR},  the splitting and gluing for stable log maps with the property of having nodes in the purely codimension one strata is given. 
Consequently, $\varphi$ is in the open\footnote{The collapse from $\Gamma$ to $\bar\Gamma$ had the purpose to achieve openness here.} image of the gluing functor, in particular, by \cite[Theorem 1.4]{KLR} the virtual fundamental class $[\s{M}_{g}(\s{X}_0^{\dagger},\Delta,\bar \Gamma)]^{\vir}$ of stable maps with collapsed tropicalization $\bar \Gamma$ locally near $\varphi$ is a rational multiple of the Gysin-pullback of the product class
\[
\prod_{V\in\bar\Gamma^{[0]}} [\s{M}_{g_V}(\s{X}_V,D_V,\bar\Gamma_V)]^{\vir}
\]
where $\s{M}_{g_V}(\s{X}_V,D_V,\bar\Gamma_V)$ denotes the (proper) moduli space of genus $g_V$ stable log maps to the component $\s{X}_V$ of $\s{X}_0$ indexed by $V$, with log structure of $\s{X}_V$ induced by its toric boundary $D_V$, and with degree $\bar\Gamma_V$.  The collapse $\Gamma'_V\ra\bar\Gamma_V$ defines a stratum in $\s{M}_{g_V}(\s{X}_V,D_V,\bar\Gamma_V)$ and by \cite[Lemma 10]{behrend-manin}, the Gysin-pullback of $[\s{M}_{g_V}(\s{X}_V,D_V,\bar\Gamma_V)]^{\vir}$ to this stratum is (a rational multiple of)
$$\prod_{V'\in{\Gamma'_V}^{[0]}} [\s{M}_{g_{V'}}(\s{X}_V,D_V,\Gamma_{V'})]^{\vir}$$
where $\Gamma_{V'}$ refers to the vertex $V'$ in $\Gamma$ with all its adjacent edges (self-adjacent edges being replaced by pairs of non-compact edges).
We want to show that (in a neighbourhood of $\varphi$), for each $V\in\Gamma^{[0]}$,
\[
\prod_{i\in I_{V}^{\circ}} \psi_i^{s_i}\cap [\s{M}_{g_V}(\s{X}_V,D_V,\Gamma_V)]^{\vir} = 0
\]
unless $\sum_{i\in I_{V}^{\circ}} s_i \leq 2g_V+\ov(V)$.  Then replacing $V$ with a genus $0$ vertex with $g_V$ additional self-adjacent edges will give a tropical curve satisfying the tropical $\psi$-class conditions at $V$, and doing this for every vertex will result in a superabundant tropical curve in $\f{T}_{g,\Delta}(\A,\Psi)$ with each $g_V$ equal to $0$ and supported on $\Gamma$.

Let $V\in \Gamma^{[0]}$ be a vertex with $g_V>0$, and let $\varphi_V^{\dagger}:C_V^{\dagger}\rar \s{X}^{\dagger}_V$ be the stable map in $\s{M}_V:=\s{M}_{g_V}(\s{X}_V,D_V,\Gamma_V)$ obtained from splitting $\varphi$ along all edges as above.
Here, $C_V$ is a curve of geometric genus $g_V$ with $\val(V)$ many marked points (self-adjacent edges give rise to two markings after splitting the node).
As in the Proposition \ref{psihat} and Lemma~\ref{PsiGeomNew}, $\psi$-classes on $\s{M}_V$ associated to marked points on $C_V$ can be pulled back from the corresponding $\?{\psi}$-classes on $\?{\s{M}}_{g_V,\val(V)}$.  It thus suffices to show that 
\begin{align}\label{PsiVir}
\forget^*\Bigg(\prod_{i\in I_V^{\circ}} \?{\psi}_i^{s_i}\Bigg)\cap [\s{M}_V]^{\vir} = 0,
\end{align} where $\forget$ is now the forgetful map $\s{M}_V\rar \?{\s{M}}_{g_V,\val(V)}$.  Since $\dim(\?{\s{M}}_{g_V,\val(V)})=3g_V+\ov(V)$, it suffices to show that $\forget_*[\s{M}_V]^{\vir}$ has support in codimension at least $g_V$.  This is what we will do.

Case I: suppose $C_V$ is contracted, i.e., all edges of $\Gamma_V$ are contracted to $V$.
It is standard that
$$[\s{M}_V]^{\vir}=-c_{\mathrm{top}}(\bb{E} \otimes \Omega_{\s{X}_V^\dagger/\Spec \kk^\dagger})\cap [\s{M}_V]$$ 
where $\bb{E}$ is the Hodge bundle, and since $\Omega_{\s{X}_V^\dagger/\Spec \kk^\dagger}$ is free as in the proof of Lemma \ref{vdim}, this is just $-c_{\mathrm{top}}(\bb{E})^n\cap [\s{M}_V]$. 
Since $C_V$ is contracted, the underlying marked curve is stable, so $\bb{E}$ is the pullback of the Hodge bundle on $\?{\s{M}}_{g_V,\val(V)}$, and so 
$\forget_*[\s{M}_V]^{\vir}$ is supported in codimension $ng$.  The left-hand side of \eqref{PsiVir} will vanish whenever $ng+\sum_{i\in I_V^{\circ}} s_i >  \dim(\?{\s{M}}_{g_V,\val(V)})=3g_V+\ov(V)$.  In particular, since $n\geq 1$, it will vanish whenever $\sum_{i\in I_V^{\circ}} s_i > 2g_V+\ov(V)$.  

Case II: now suppose $C_V$ is not contracted, i.e., some edges of $\Gamma_V$ are not contracted.  Choose an $m\in M$ such that $\langle m,u_E\rangle \neq 0$ for some $E\ni V$.  Then $\varphi_V^*(z^m)$ is a rational function on $C_V$ whose zeroes and poles are at the nodes and marked points $y_E$ corresponding to edges $E\ni V$, with the order of the zeroes and poles being given by $\langle m,w(E)u_{(V,E)} \rangle$.  Hence, $\sum_{E\ni V} \langle m,w(E)u_{(V,E)} \rangle y_E$ determines a degree $0$ divisor on $C_V$ whose corresponding line bundle must be trivial.  Since $\Pic_0(C_V)$ is $g_V$-dimensional, this determines a codimension $g_V$ condition on $\?{\s{M}}_{g_V,\val(V)}$ that $C_V$ must satisfy.  Hence, the left-hand side of \eqref{PsiVir} will indeed vanish whenever $\sum_{i\in I_V^{\circ}} s_i > 2g_V+\ov(V)$.
\qed

\bibliographystyle{amsalpha}  % Here the bibliography 		     %
\bibliography{biblio}        % is inserted.			     %
\index{Bibliography@\emph{Bibliography}}%

\end{document}

%% file: glossary.tex
\newglossaryentry{kk}{
name={\ensuremath{\kk}},
symbol={\ensuremath{\kk}},
description={Fixed algebraically closed field of characteristic zero}}

\newglossaryentry{N}{
name={\ensuremath{N}},
description={Free Abelian group of finite rank $n$}}

\newglossaryentry{LA}{
name={\ensuremath{\LL(A)}},
description={Linear subspace of $N_{\bb{Q}}$, translate of the affine subspace $A\subset N_{\bb{Q}}$ to the origin}}

\newglossaryentry{LNA}{
name={\ensuremath{\LL_N(A)}},
description={$\LL(A)\cap N$}}

\newglossaryentry{M}{
name={\ensuremath{M}},
description={Dual lattice of $N$}}

\newglossaryentry{N'}{
name={\ensuremath{N'}},
description={$N\oplus \bb{Z}$}}

\newglossaryentry{M'}{
name={\ensuremath{M'}},
description={Dual lattice of $N'$}}

\newglossaryentry{CA}{
name={\ensuremath{\C(A)}},
description={Closed cone $\?{\{(ta,t)|a\in A,t\in \bb{Q}_{\geq 0}\}}\subset N'_{\bb{Q}}$ over $A$}}

\newglossaryentry{LCA}{
name={\ensuremath{\LC(A)}},
description={$\LL(\C(A))$}}

\newglossaryentry{Gammabar}{
name={\ensuremath{\overline{\Gamma}}},
description={Topological realization of a finite connected graph}}

\newglossaryentry{Gamma}{
name={\ensuremath{\Gamma}},
description={Complement of some subset of the $1$-valent vertices of $\?{\Gamma}$}}

\newglossaryentry{Gamma0}{
name={\ensuremath{\Gamma^{[0]}}},
description={Vertices of $\Gamma$}}

\newglossaryentry{Gamma1}{
name={\ensuremath{\Gamma^{[1]}}},
description={Edges of $\Gamma$}}

\newglossaryentry{Gamma1inf}{
name={\ensuremath{\Gamma^{[1]}_{\infty}}},
description={Non-compact edges of $\Gamma$}}

\newglossaryentry{Gamma1c}{
name={\ensuremath{\Gamma^{[1]}_c}},
description={Compact edges of $\Gamma$}}

\newglossaryentry{w}{
name={\ensuremath{w}},
description={Weight-function $\Gamma^{[1]}\rar \bb{Z}_{\geq 0}$}}

\newglossaryentry{g}{
name={\ensuremath{g}},
description={Genus-function $\Gamma^{[0]}\rar \bb{Z}_{\geq 0}$. We typically write $g(v)$ as $g_V$}}

\newglossaryentry{eps}{
name={\ensuremath{\epsilon}},
description={Marking $I\risom \Gamma^{[1]}_{\infty}$ for some index-set $I$.  We often write $\epsilon(i)$ as $E_i$}}

\newglossaryentry{Icirc}{
name={\ensuremath{I^{\circ}}},
description={Set of $i\in I$ for which $w(E_i)=0$.  I.e., the interior markings}}

\newglossaryentry{Ip}{
name={\ensuremath{I^{\partial}}},
description={Set of $i\in I$ for which $w(E_i)\neq 0$.  I.e., the boundary markings}}

\newglossaryentry{Gammaeps}{
name={\ensuremath{(\Gamma,\epsilon)}},
description={Data of $\Gamma$, the weight-function $w$ and genus-function $g$, and the marking $\epsilon$}}

\newglossaryentry{einf}{
name={\ensuremath{e_{\infty}}},
description={$\#I$, or equivalently, $\#\Gamma^{[1]}_{\infty}$}}

\newglossaryentry{IVcirc}{
name={\ensuremath{I_V^{\circ}}},
description={Set of $i\in I^{\circ}$ such that $E_i$ contains the vertex $V$}}

\newglossaryentry{mV}{
name={\ensuremath{m_V}},
description={$\# I_V^{\circ}$, i.e., the number of weight-zero non-compact edges containing $V$}}

\newglossaryentry{ptc}{
name={\ensuremath{(\Gamma,\epsilon,h)}},
description={Parametrized marked tropical curve, or the (marked) tropical curve it represents}}

\newglossaryentry{h}{
name={\ensuremath{h}},
description={Map $\Gamma \rar N_{\bb{R}}$ from the data of a parametrized marked tropical curve}}

\newglossaryentry{uVE}{
name={\ensuremath{u_{(V,E)}}},
description={Primitive integral vector emanating from $h(V)$ into $h(E)$, or $0$ if $h(E)$ is a point}}

\newglossaryentry{ui}{
name={\ensuremath{u_i}},
description={Simplified notation for $u_{(V,E_i)}$, where $V$ is the unique vertex of the non-compact edge $E_i$
}}

\newglossaryentry{uE}{
name={\ensuremath{u_{E}}},
description={Simplified notation for $u_{(V,E)}$ when the choice of $V\in E$ is clear from context or unimportant
}}

\newglossaryentry{AutG}{
name={\ensuremath{\Aut(\Gamma)}},
description={Automorphism group of the tropical curve $(\Gamma,\epsilon,h)$
}}

\newglossaryentry{gGamma}{
name={\ensuremath{g(\Gamma)}},
description={Genus of $\Gamma$, defined as $b_1(\Gamma)+\sum_{V\in \Gamma^{[0]}} g_V$
}}

\newglossaryentry{ebar}{
name={\ensuremath{\overline{e}}},
description={$\# \Gamma^{[1]}_c$, i.e., the number of compact edges of $\Gamma$
}}

\newglossaryentry{valV}{
name={\ensuremath{\val(V)}},
description={Valence of a vertex $V$, counting self-adjacent edges twice
}}

\newglossaryentry{ovV}{
name={\ensuremath{\ov(V)}},
description={Over-valence $\val(V)-3$ of a vertex $V$
}}

\newglossaryentry{ovG}{
name={\ensuremath{\ov(\Gamma)}},
description={Over-valence $\sum_{V\in \Gamma^{[0]}} \ov(V)$ of $\Gamma$
}}

\newglossaryentry{flags}{
name={\ensuremath{F\Gamma}},
description={Set of flags $(V,E)$ of $\Gamma$
}}

\newglossaryentry{typespace}{
name={\ensuremath{\f{T}_{(\Gamma,\epsilon,u)}}},
description={Space of marked tropical curves of type $(\Gamma,\epsilon,u)$
}}

\newglossaryentry{Delta}{
name={\ensuremath{\Delta}},
description={Degree of a type $(\Gamma,\epsilon,u)$, i.e., the map taking $i\in I$ to $\Delta(i):=w(E_i)u_i\in N$
}}

\newglossaryentry{AA}{
name={\ensuremath{\A}},
description={An affine constraint, i.e., a tuple $(A_i)_{i\in I}$ of affine subspaces of $N_{\bb{Q}}$
}}

\newglossaryentry{Psi}{
name={\ensuremath{\Psi}},
description={A tuple $(s_i)_{i\in I^{\circ}}\in \bb{Z}_{\geq 0}^{|I^{\circ}|}$ indicating the $\psi$-class conditions
}}

\newglossaryentry{TAPsi}{
name={\ensuremath{\f{T}_{g,\Delta}(\A,\Psi)}},
description={Space of tropical curves of genus $g$ and degree $\Delta$ satisfying $\A$ and $\Psi$
}}

\newglossaryentry{<V>}{
name={\ensuremath{\langle V \rangle}},
description={$\ov(V)!/(\prod_{i\in I_V^{\circ}} s_i!)$
}}

\newglossaryentry{PhiQ}{
name={\ensuremath{\Phi_{\bb{Q}}}},
description={Map given in \eqref{D0}
}}

\newglossaryentry{dtrop}{
name={\ensuremath{d^{\trop}_{g,\Delta}}},
description={Expected tropical dimension $e_{\infty}+(n-3)(1-g)$
}}

\newglossaryentry{Tns}{
name={\ensuremath{\f{T}^{\ns}_{g,\Delta}(\A,\Psi)}},
description={Non-superabundant elements of $\f{T}_{g,\Delta}(\A,\Psi)$}
}

\newglossaryentry{Phi}{
name={\ensuremath{\Phi}},
description={Map of lattices given in \eqref{D}}
}

\newglossaryentry{DGamma}{
name={\ensuremath{\f{D}_{\Gamma}}},
description={$\inde(\Phi) \prod_{V\in \Gamma^{[0]}} \langle V \rangle$}
}

\newglossaryentry{Mult}{
name={\ensuremath{\Mult(\Gamma)}},
description={$\frac{\f{D}_{\Gamma}}{|\Aut(\Gamma)|} \prod_{E\in \Gamma^{[1]}_c} w(E)$}
}

\newglossaryentry{GWtrop}{
name={\ensuremath{{}^{\ns}\GW_{g,N_{\bb{Q}},\Delta}^{\trop}(\A,\Psi)}},
description={$\sum_{(\Gamma,\epsilon,h)\in \f{T}^{\ns}_{g,\Delta}(\A,\Psi)} \Mult(\Gamma)$, denoted $\GW_{g,N_{\bb{Q}},\Delta}^{\trop}(\A,\Psi)$ if $\f{T}_{g,\Delta}(\A,\Psi)$ has no superabundant curves}
}

\newglossaryentry{P}{
name={\ensuremath{\s{P}}},
description={A polyhedral decomposition of $N_{\bb{Q}}$, typically a ``good'' one $\s{P}_{g,\Delta}(\A,\Psi)$ as in Lemma \ref{lem-good-subd}}
}

\newglossaryentry{SigmaPprime}{
name={\ensuremath{\Sigma'_{\s{P}}}},
description={Fan given as the cone over $\s{P}$}
}

\newglossaryentry{sX}{
name={\ensuremath{\s{X}}},
description={Toric family over $\bb{A}^1$ with fan $\Sigma'_{\s{P}}$}
}

\newglossaryentry{Xt}{
name={\ensuremath{\s{X}_t}},
description={Fiber of $\s{X}$ over $t\in \bb{A}^1$}
}

\newglossaryentry{SigmaP}{
name={\ensuremath{\Sigma_{\s{P}}}},
description={Asymptotic fan of $\s{P}$}
}

\newglossaryentry{X}{
name={\ensuremath{X}},
description={Toric variety with fan $\Sigma_{\s{P}}$}
}

\newglossaryentry{MYS}{
name={\ensuremath{\s{M}(Y^{\dagger}/S^{\dagger},\beta)}},
description={Moduli of basic stable log maps to a log scheme $Y^{\dagger}$ over another log scheme $S^{\dagger}$, satisfying a combinatorially finite collection of conditions $\beta$}
}

\newglossaryentry{[Delta]}{
name={\ensuremath{[\Delta]}},
description={Curve class determined by $\Delta$}
}

\newglossaryentry{xi}{
name={\ensuremath{x_i}},
description={Interior marked point}
}

\newglossaryentry{yi}{
name={\ensuremath{y_i}},
description={Boundary marked point}
}

\newglossaryentry{Mtt}{
name={\ensuremath{\s{M}^{tt}_{g}(\s{X}_t^{\dagger},\Delta)}},
description={Moduli of torically transverse curves}
}

\newglossaryentry{ZA}{
name={\ensuremath{Z_{A}}},
description={Subvariety representing the incidence condition corresponding to the tropical condition $A$}
}

\newglossaryentry{ZAu}{
name={\ensuremath{Z_{A,u}}},
description={$Z_A\cap D_u$ for $D_u$ the boundary divisor corresponding to $u\in N$.  Represents a boundary incidence condition}
}

\newglossaryentry{psii}{
name={\ensuremath{\psi_i}},
description={Psi-class condition $c_1(\sigma_{x_i}^* \omega_{\pi})$}
}

\newglossaryentry{psiibar}{
name={\ensuremath{\?{\psi}_{i}}},
description={Psi-class condition on $\?{\s{M}}_{g,e_{\infty}}$}
}

\newglossaryentry{psiihat}{
name={\ensuremath{\wh{\psi}_{i}}},
description={$\forget^*\?{\psi}_i$}
}

\newglossaryentry{gamma}{
name={\ensuremath{\gamma}},
description={Gromov-Witten cycle associated to $\A$ and $\Psi$}
}

\newglossaryentry{GWlog}{
name={\ensuremath{\GW^{\log}_{g,\s{X}^{\dagger}_t,\Delta}(\A,\Psi)}},
description={Log Gromov-Witten number}
}

\newglossaryentry{nsGWlog}{
name={\ensuremath{{}^{\ns}\GW^{\log}_{g,\s{X}^{\dagger}_0,\Delta}(\A,\Psi)}},
description={Non-superabundant contribution to the log Gromov-Witten number}
}

\newglossaryentry{kdag}{
name={\ensuremath{\Spec \kk^{\dagger}}},
description={Standard log point}
}

\newglossaryentry{Q}{
name={\ensuremath{Q}},
description={Basic monoid}
}

\newglossaryentry{kbas}{
name={\ensuremath{\Spec \kk_{\varphi}^{\dagger}}},
description={Basic log point associated to $\varphi$}
}

\newglossaryentry{Peta}{
name={\ensuremath{P_{\eta}}},
description={ghost sheaf stalk at the image under $\varphi$ of the generic point $\eta$ of a curve component}
}

\newglossaryentry{Tphi}{
name={\ensuremath{\f{T}_{\varphi}}},
description={The space of primitive tropical curves associated to $\varphi$}
}

\newglossaryentry{wtMG}{
name={\ensuremath{\wt{\s{M}}(\Gamma)}},
description={$\prod_{V\in \Gamma^{[0]}} \?{\s{M}}_{g_V,\val(V)}$}
}

\newglossaryentry{MGam}{
name={\ensuremath{\s{M}(\Gamma)}},
description={$[\wt{\s{M}}(\Gamma)/\Aut(\Gamma)]$}
}

\newglossaryentry{Mprelog}{
name={\ensuremath{\s{M}^{\mathrm{pre-log}}_{g}(\s{X}_0^{\dagger},\Delta,\Gamma)}},
description={Space of pre-log curves}
}